\documentclass[10pt]{amsart}

\usepackage{mathtools}
\usepackage{amsthm}
\usepackage{amssymb}
\usepackage[all]{xy}
\usepackage[nobottomtitles*]{titlesec}
\usepackage{titletoc}
\usepackage{bbm}
\usepackage{upgreek}
\usepackage{ marvosym }
\usepackage{lscape}
\usepackage{ifluatex}
\usepackage[pdfencoding=auto, psdextra]{hyperref}
\usepackage[noabbrev,nameinlink]{cleveref}
\usepackage{enumerate}
\usepackage{soul}
\usepackage{stmaryrd}

\usepackage{verbatim}
\usepackage{tikz-cd}
\makeatletter
\tikzcdset{
    cong/.style={"\cong" {sloped, description, yshift=0pt,#1}, phantom},
    simeq/.style={"\simeq" {sloped, description, yshift=0pt,#1}, phantom},
    snake/.style={
        out= east, in=west,
        to path={
            \pgfextra{
                \pgfextractx{\pgf@xa}{\pgfpointanchor{\tikztostart}{east}}
                \pgfextractx{\pgf@xb}{\pgfpointanchor{\tikztotarget}{west}}
                \pgfextracty{\pgf@ya}{\pgfpointanchor{\tikztostart}{center}}
                \pgfextracty{\pgf@yb}{\pgfpointanchor{\tikztotarget}{center}}
                \edef\tikzstartx{\the\pgf@xa}
                \edef\tikzendx{\the\pgf@xb}
                \edef\midy{\the\dimexpr0.5\dimexpr\pgf@ya\relax +0.5\dimexpr\pgf@yb\relax}
            }
            to[in=0,out=180,looseness=0.5] (\tikzstartx,\midy)
            -| ([xshift=-2ex]\tikztotarget.west)
            -- (\tikztotarget)}
    }
}
\makeatother
\usepackage{tikzcdintertext}
\usepackage{etoolbox}
\usepackage{xpatch}

\crefname{diagram}{diagram}{diagrams}
\creflabelformat{diagram}{#2(#1)#3}
\crefname{sseq}{}{}
\crefname{equation}{}{}
\crefname{equation-b}{Equation}{Equations}
\creflabelformat{sseq}{#2(#1)#3}
\usepackage[normalem]{ulem}
\newcommand{\stkout}[1]{\ifmmode\text{\sout{\ensuremath{#1}}}\else\sout{#1}\fi}

\makeatletter

\pretocmd{\maketitle}{%
    \edef\@title{\unexpanded{\protect\Large}\unexpanded\expandafter{\@title}}%
    \edef\authors{\unexpanded{\protect\normalsize}\unexpanded\expandafter{\authors}}%
}{}{\error}

\def\addlabeltolink#1{{\addlabeltolink@#1}}
\def\addlabeltolink@#1#2#3#4{#1{#2}{#3}{\thecontentslabel. #4}}
\contentsmargin{-1.5em}
\titlecontents{section}[1em]{}{\addlabeltolink}{\hspace*{-1em}}{
\titlerule*[1pc]{.}\hphantom{11.}\llap{\thecontentspage}}

\titleformat{\section}[block]{\centering\bfseries\Large}{\thetitle. }{0pt}{}
\titlespacing{\section}{0pt}{*4}{*2}
\titleformat{\subsection}[hang]{\bfseries}{\thetitle. }{0pt}{}
\titlespacing{\subsection}{0pt}{*2}{*1.5}

\def\@secnumpunct{. }
\xpatchcmd{\proof}{\topsep6\p@\@plus6\p@\relax}{\topsep0pt\relax}{}{\error}
\makeatother
\usepackage[parfill]{parskip}

\allowdisplaybreaks[1]
\newcommand{\susp}{\Sigma}

\newcommand{\suspinfty}{\Sigma^{\infty}}

\newcommand{\mmod}{\! \sslash \!}

\newcommand\Einfty{\mathbb{E}_{\infty}}

\def\F{\mathbb{F}}

\newcommand{\mr}[1]{\mathrm{#1}}

\newcommand{\mfrak}{\mathfrak{m}}

\newcommand{\Fp}{\FF_p}
\newcommand{\Qp}{\QQ_p}

\newcommand{\Gmhat}{\widehat{\mathbb{G}}_m}

\newcommand{\Gammatwee}{\widetilde{\Gamma}}

\newcommand{\CP}{\mathbb{CP}}

\newcommand{\CPinfty}{\CP^{\infty}}

\newcommand{\order}{\Theta}

\newcommand{\Sp}{\mathbb{S}}

\newcommand{\HZ}{\mr{H}\ZZ}

\newcommand{\R}{\mr{R}}

\newcommand{\HF}{\mr{H}\mathbb{F}}

\newcommand\HFp{\HF_p}
\newcommand\FFtwee{\widetilde{\FF}}

\newcommand{\BU}{\mr{BU}}

\newcommand{\MU}{\mr{MU}}
\newcommand{\MUn}{\MU[n]}

\newcommand{\KU}{\mr{KU}}
\newcommand{\KO}{\mr{KO}}
\newcommand{\KUhat}{\KU^{\raisebox{0.2ex}{$\scriptscriptstyle\wedge$}}}
\newcommand{\KOhat}{\KO^{\raisebox{0.1ex}{$\scriptscriptstyle\wedge$}}}
\newcommand{\Ects}{\E^{\raisebox{0.35ex}{$\scriptscriptstyle\wedge$}}}
\newcommand{\MapsCts}{\Maps^{\mathrm{cts}}}
\DeclareMathOperator{\Maps}{Maps}
\newcommand{\BO}{\mr{BO}}
\newcommand{\EO}{\mr{EO}}

\newcommand\E{\mr{E}}

\newcommand\ER{\mr{ER}}
\newcommand{\EEO}{\mr{E}^{\EO}}
\newcommand{\DEO}{\mr{D}_{\EO}}
\newcommand{\KEO}{\mr{K}^{\EO}}

\newcommand{\M}{\mr{M}}
\newcommand{\Ctw}{\mr{C}_2}
\newcommand{\Cp}{\mr{C}_{p}}

\newcommand\hCp{\mr{hC}_p}

\newcommand\tSo{\mr{tS}^1}

\newcommand\hG{\mr{hG}}
\newcommand\tG{\mr{tG}}

\newcommand{\bk}{{\upbeta_k}}
\newcommand{\obk}{{\hat{\upbeta}_k}}

\newcommand\EOmod{\mathop{\EO\textup{--}\mathrm{Mod}}}
\newcommand\EEOEcomod{\Hom_{\EEO_*\E}}

\newcommand\ECp{\E_*[\Cp]^{\sigma}}
\newcommand\ECpsigma{\E_*[\Cp]^{\sigma}}
\newcommand\KCp{\K_*[\Cp]}
\newcommand\grKCp{\gr \K_*[\Cp]}

\newcommand\grK{\gr\K}

\newcommand{\BP}{\mr{BP}}
\newcommand{\K}{\mr{K}}
\newcommand{\LK}{\Lmr_{\mr{K}}}

\newcommand{\A}{\mathcal{A}}
\newcommand{\orderB}[1]{\order( #1 )}

\newcommand{\sk}{{\sf sk}}
\newcommand{\cosk}{{\sf cosk}}

\hypersetup{%
  bookmarksnumbered=true,%
  bookmarks=true,%
  colorlinks=true,%
  pdfnewwindow=true,%
  pdfstartview=FitBH%
}

\hypersetup{
    colorlinks,
    linkcolor={red!30!brown!70!black},
    citecolor={blue!70!black},
    urlcolor={blue!80!black}
}

\def\@url#1{{\tt\def~{\lower3.5pt\hbox{\char'176}}\def\_{\char'137}#1}}

\makeatletter
\let\c@lemma\c@theorem
\makeatother

\newtheorem{thm}[equation]{Theorem}
\newtheorem{main}[equation]{Main Theorem}
\newtheorem{cor}[equation]{Corollary}
\newtheorem{lem}[equation]{Lemma}

\newtheorem{prop}[equation]{Proposition}

\newtheorem{conjecture}[equation]{Conjecture}

\theoremstyle{definition}
\newtheorem{defn}[equation]{Definition}

\newtheorem{ex}[equation]{Example}

\newtheorem{rmk}[equation]{Remark}
\newtheorem{notation}[equation]{Notation}

\newtheorem*{thm*}{Theorem}
\newtheorem*{cor*}{Corollary}
\newtheorem*{lem*}{Lemma}
\newtheorem*{prop*}{Proposition}
\newtheorem*{not*}{Notation}
\newtheorem*{guess*}{Guess}

\newtheorem*{defn*}{Definition}
\newtheorem*{ex*}{Example}
\newtheorem*{exs*}{Examples}
\newtheorem*{rmk*}{Remark}
\newtheorem*{claim*}{Claim}
\newtheorem*{exer*}{Exercise}

\numberwithin{equation}{section}
\numberwithin{figure}{section}

\makeatletter
\let\c@lem=\c@thm
\let\c@cor=\c@thm
\let\c@prop=\c@thm
\let\c@lem=\c@thm
\let\c@ex=\c@thm
\let\c@exs=\c@thm
\let\c@obs=\c@thm
\let\c@rmk=\c@thm
\let\c@perthm=\c@thm
\let\c@conjtel=\c@thm
\let\c@exmps=\c@thm
\let\c@rem=\c@thm
\let\c@question=\c@thm
\let\c@warn=\c@thm
\let\c@claim=\c@thm
\let\c@quest=\c@thm
\let\c@notation=\c@thm
\let\c@note=\c@thm
\let\c@conjtel=\c@thm
\let\c@gue=\c@thm
\let\c@goal=\c@thm
\makeatother

\DeclareMathOperator{\Ext}{Ext}
\DeclareMathOperator{\Hom}{Hom}

\DeclareMathOperator{\Aut}{Aut}

\DeclareMathOperator{\End}{End}
\DeclareMathOperator{\Fil}{Fil}
\DeclareMathOperator{\gr}{\smash{\mathrm{gr}}}
\DeclareMathOperator{\Res}{Res}

\DeclareMathOperator{\Map}{Map}

\DeclareMathOperator*{\colim}{colim}

\DeclareMathOperator{\Th}{Th} 

\DeclareMathOperator{\GL}{GL_1}
\DeclareMathOperator{\BGL}{BGL_1}

\DeclareMathOperator{\id}{id}
\DeclareMathOperator{\ev}{ev}

\renewcommand\H{\mathrm{H}}

\DeclareMathOperator\cofiber{Cofib}

\def\makecommands#1#2#3{
    \bgroup
    \def\tempcmdname##1{#1}
    \def\tempcmdbody##1{#2}
    \def\\##1{\expandafter\xdef\csname\tempcmdname{##1}\endcsname{\unexpanded\expandafter{\tempcmdbody{##1}}}}
    #3
    \egroup
}
\def\upperalphabet{\\A\\B\\C\\D\\E\\F\\G\\H\\I\\J\\K\\L\\M\\N\\O\\P\\Q\\R\\S\\T\\U\\V\\W\\X\\Y\\Z}
\def\loweralphabet{\\a\\b\\c\\d\\e\\f\\g\\h\\i\\j\\k\\l\\m\\n\\o\\p\\q\\r\\s\\t\\u\\v\\w\\x\\y\\z}
\def\lowergreekalphabet{\\\alpha\\\beta\\\gamma\\\delta\\\epsilon\\\zeta\\\eta\\\theta\\\kappa\\\lambda\\\mu\\\nu
    \\\xi\\\pi\\\rho\\\sigma\\\tau\\\upsilon\\\psi\\\chi\\\phi\\\omega}

\makecommands{#1mr}{\mathrm{#1}}{\upperalphabet}
\makecommands{#1cal}{\mathcal{#1}}{\upperalphabet}
\makecommands{#1#1}{\mathbb{#1}}{\upperalphabet}
\makecommands{#1bar}{\overline{#1}}{\upperalphabet\loweralphabet}
\makecommands{#1twee}{\widetilde{#1}}{\upperalphabet\loweralphabet}
\makecommands{sf#1}{{\sf #1}}{\upperalphabet\loweralphabet}
\makeatletter
\makecommands{\expandafter\@gobble\string#1bar}{\overline{#1}}{\lowergreekalphabet}
\makecommands{\expandafter\@gobble\string#1twee}{\widetilde{#1}}{\lowergreekalphabet}
\makeatother

\newcommand\sfsbar{\overline{\sfs}}

\newcommand\limone{\lim\nolimits^1}
\def\rcolon{\nobreak {}\nonscript\mskip 6muplus1mu{:} \mskip 2mu\mathpunct\relax }

\newcommand{\sma}{\wedge}
\newcommand{\sm}{\wedge}

\definecolor{limegreen}{rgb}{0.2, 0.8, 0.2}
\definecolor{darkmagenta}{rgb}{0.55, 0.0, 0.55}
\definecolor{lavenderrose}{rgb}{0.91, 0.33, 0.5}
\definecolor{goldenpoppy}{rgb}{0.99, 0.76, 0.0}
\definecolor{seagreen}{rgb}{0.11, 0.35, 0.02}
\definecolor{maroon}{RGB}{128,0,0}
\definecolor{darkviolet}{RGB}{148,0,211}

\makeatletter

\def\tikzcdequalsignoffset{0.1em}

\def\findedgesourcetarget#1#2{
    \let\sourcecoordinate\pgfutil@empty
    \ifx\tikzcd@startanchor\pgfutil@empty 
        \def\tempa{\pgfpointanchor{#1}{center}}
    \else
        \edef\tempa{\noexpand\pgfpointanchor{#1}{\expandafter\@gobble\tikzcd@startanchor}} 
        \let\sourcecoordinate\tempa
    \fi
    \ifx\tikzcd@endanchor\pgfutil@empty 
        \def\tempb{\pgfpointshapeborder{#2}{\tempa}}
    \else
        \edef\tempb{\noexpand\pgfpointanchor{#2}{\expandafter\@gobble\tikzcd@endanchor}}
    \fi
    \let\targetcoordinate\tempb
    \ifx\sourcecoordinate\pgfutil@empty%
        \def\sourcecoordinate{\pgfpointshapeborder{#1}{\tempb}}%
    \fi
}

\tikzset{/tikz/commutative diagrams/equal/.style=equals,
    /tikz/commutative diagrams/equals/.style = {
    -,
    to path={\pgfextra{
        \findedgesourcetarget{\tikzcd@ar@start}{\tikzcd@ar@target} 
        \ifx\tikzcd@startanchor\pgfutil@empty
            \def\tikzcd@startanchor{.center}
        \fi
        \ifx\tikzcd@endanchor\pgfutil@empty
            \def\tikzcd@endanchor{.center}
        \fi
        \pgfmathanglebetweenpoints{\pgfpointanchor{\tikzcd@ar@start}{\expandafter\@gobble\tikzcd@startanchor}}{\pgfpointanchor{\tikzcd@ar@target}{\expandafter\@gobble\tikzcd@endanchor}}
        \pgftransformrotate{\pgfmathresult}
        \pgfpathmoveto{\pgfpointadd{\sourcecoordinate}{\pgfpoint{0}{\tikzcdequalsignoffset}}}
        \pgfpathlineto{\pgfpointadd{\targetcoordinate}{\pgfpoint{0}{\tikzcdequalsignoffset}}}
        \pgfpathmoveto{\pgfpointadd{\sourcecoordinate}{\pgfpoint{0}{-\tikzcdequalsignoffset}}}
        \pgfpathlineto{\pgfpointadd{\targetcoordinate}{\pgfpoint{0}{-\tikzcdequalsignoffset}}}
        \pgfusepath{draw}
}}}}

\makeatother

\title{On the $\EO$-orientability of vector bundles}

\author{P.~Bhattacharya}\address{University of Virginia}\email{pb9wh@virginia.edu}
\author{H.~Chatham}\address{Massachusetts Institute of Technology }\email{hood@mit.edu}

\thanks{}

\begin{document}

\begin{abstract}
We study the orientability of vector bundles with respect to a family of
cohomology theories called $\mathrm{EO}$-theories. The $\mathrm{EO}$-theories
are higher height analogues of real $\mathrm{K}$-theory $\mathrm{KO}$. For each
$\mathrm{EO}$-theory, we prove that the direct sum of $i$ copies of any vector
bundle is $\mathrm{EO}$-orientable for some specific integer $i$. Using a
splitting principal, we reduce to the case of the canonical line bundle over
$\mathbb{CP}^{\infty}$. Our method involves understanding the action of an order
$p$ subgroup of the Morava stabilizer group on the Morava $\mathrm{E}$-theory of
$\mathbb{CP}^{\infty}$. Our calculations have another application: We determine
the homotopy type of the $\mathrm{S}^{1}$-Tate spectrum associated to the
trivial action of $\mathrm{S}^{1}$ on all $\mathrm{EO}$-theories.
\end{abstract}

\maketitle	


\section{Introduction}
\label{sec:intro}
A real vector bundle $\upxi$ is called orientable with respect to a cohomology
theory $\mr{E}$ if it cannot distinguish between $\upxi$ and the trivial vector
bundle of the same dimension over the same base space, which we will denote by
$\upepsilon^{\dim \upxi}$. In other words, there is an $\mr{E}$-Thom isomorphism
\[ 
\mr{E} \sma \mr{Th}(\upxi)  \simeq \mr{E} \sma  \mr{Th}(\upepsilon^{\dim \upxi}) 
\]
where $\mr{Th}(\upxi)$ denotes the Thom space associated to $\upxi$ (see
\eqref{Thomspace}). The study of the orientability of vector bundles with
respect to cohomology theories has led to  foundational results in both
geometry and algebraic topology.  
For example, the $\mr{H}\F_2$-orientability of all vector bundles leads to the
notion of Stiefel-Whitney classes and the $\mr{H}\ZZ$-orientability of complex
vector bundles leads to Chern classes.  It is a standard fact that all complex
vector bundles are $\KU$-orientable, however, not all of them are
$\KO$-orientable. The following examples motivate some of the work in this
paper.
\begin{ex} 
The complex tautological line bundle $\upgamma^{1}$ over $\CP^1 \simeq \mr{S}^2$
is not $\KO$-orientable because the Thom isomorphism fails:
\[ \KO \sma  \Th(\upgamma^{1})  \not\simeq  \KO \sma \mr{Th}(\upepsilon^2)  .\]
On one hand,  $\Th(\upgamma^{1})\simeq \CP^2$ (see \Cref{thomCP})  and $  \KO \sma \CP^2  \simeq \KU$.
 On the other hand, $\mr{Th}(\upepsilon^2) \simeq \Sigma^2
\CP^1_+ \simeq \mr{S}^2 \vee \mr{S}^4$ (see \Cref{ex:Thomtriv}) and 
$ \KO \sm \mr{Th}(\upepsilon^2) \simeq \Sigma^2 \KO \vee \Sigma^4 \KO$. 
\end{ex}
\begin{ex}
While $\upgamma^1$ is not $\KO$-orientable, its $2$-fold direct sum
$\upgamma^{1} \oplus \upgamma^{1}$ is. In fact, $\upxi^{\oplus 2}$ is
$\KO$-orientable for every complex vector bundle $\upxi$ (see
\Cref{splitting-principal}).
\end{ex}
From the point of view of chromatic homotopy theory, $2$-completed $\KO$ is part
of a family of cohomology theories called the $\EO$-theories.  Associated to
each finite height formal group $\Gamma$ over a perfect field $\FF$ of
characteristic $p$, there is an associated Morava $\E$-theory $\E_{\Gamma}$ with
an action of $\Aut(\Gamma)$, often called the small Morava Stabilizer group.
When  the cyclic group of order $p$ acts faithfully on $\Gamma$, we define the spectrum $\EO_{\Gamma}$ as the homotopy
fixed point spectrum \[ \EO_{\Gamma} \coloneqq \E^{\hCp}_{\Gamma}.\] 
\begin{ex}
When $p=2$ and $\Gmhat$ is the multiplicative formal group over
$\FF_2$, there is a canonical action of $\Ctw$ on $\Gmhat$. In this case,
$\E_{\Gmhat} \simeq \KUhat_2$ and $\EO_{\Gmhat} \simeq \KOhat_2$. 
\end{ex}
\begin{rmk} \label{heightrestrict} Let $n$ denote the height of the formal group
$\Gamma$.  If $p-1$ divides $n$  and the base field $\FF$ of $\Gamma$ is
algebraically closed, then $\Aut(\Gamma)$ contains a subgroup of order $p$ which
is unique up to conjugation. If $\FF$ is not algebraically closed  then any
number of conjugacy classes of subgroups of order $p$ (including zero) is
possible depending on $\Gamma$. If $p-1$ does not divide $n$, then
$\Aut(\Gamma)$ does not contain any subgroup of order $p$ (see
\cite{Serre-LCFT}). 
\end{rmk}
The main goal of this paper is to prove the following result:
\begin{main}
\label{main1} Let $p$ be a prime and $k>0$. The $p^{p^{k} -1}$-fold direct sum
of any $\CC$-vector bundle is $\EO_{\Gamma}$-orientable if the height of  formal
group $\Gamma$ is  $(p-1)k$. 
 \end{main}

For  a ring spectrum  $\R$, the \emph{$\R$-orientation
order} of a vector bundle $\upxi$, denoted by $\orderB{\R, \upxi}$, is the
smallest positive integer $d$ for which the $d$-fold direct sum is
$\mr{R}$-orientable (see \Cref{defn:orientorder} and \Cref{rmk:geointerpret}). 

The proof of \Cref{main1} reduces to the study of a single vector bundle, namely the tautological line bundle $\upgamma$ over
$\CPinfty$, because of a splitting principal which asserts that 
$\orderB{\upxi, \R}$ divides $\orderB{\upgamma, \R}$ for any $\CC$-vector bundle $\upxi$ (see \Cref{splitting-principal}). Thus, \Cref{main1} follows from the following result.
\begin{main} \label{main1.5}  $\orderB{ \upgamma, \EO_{\Gamma}}$ divides $p^{p^k -1}$. 
\end{main}

\Cref{main1.5} is  a consequence of the study of the action of
$\mr{C}_p \subset \Aut(\Gamma)$ on $\mr{E}_{\Gamma}^{\ *}\CP^\infty$. We first
show that  $\mr{E}_{\Gamma}^{\ *}\CP^\infty$ admits a ``$\text{Free}\oplus\text{Finite}$'' decomposition as
a $\Cp$-module (see \Cref{ECPinfty-decomposition}). Then, using the relative Adams spectral sequence
\cite{BakerLazarev} we lift this algebraic splitting to obtain the following
$\EO_{\Gamma}$-module splitting:
\begin{main}
\label{main2} In the category of $\EO_{\Gamma}$-modules there is a splitting 
\[ \EO_{\Gamma} \sma \CPinfty_+ \simeq \Mcal\vee \Fcal, \] 
where  $\Mcal$ is a compact $\EO_{\Gamma}$-module and $\Fcal\simeq \bigvee_{\NN} \E_{\Gamma}$ is a wedge of Morava $\E$-theories.  
\end{main}
 Further, we show that the inclusion map
 \[
 \begin{tikzcd}
 \Mcal \rar[] & \EO_{\Gamma} \sm \CPinfty_+
 \end{tikzcd} 
 \]
 of the compact summand  factors through $\EO_{\Gamma} \sm \CP^{p^{k+1} - p^k
 -1}_+$ (see \Cref{M-factor-skeleton}). Consequently 
\[ \Theta(\EO_{\Gamma}, \upgamma) = \Theta(\EO_{\Gamma}, \upgamma^{p^{k+1}-p^k-1}),\] 
where $\upgamma^d$ is the restriction of $\upgamma$ to $\CP^d$ (see
\Cref{cor:order-reduce}).  Atiyah and Todd \cite{AtiyahTodd} computed
$\Theta(\SS, \upgamma^{d})$ for all $d$ which serves as an upper bound for
$\Theta(\EO_{\Gamma}, \upgamma^{d})$ and leads to \Cref{main1.5} (see
\Cref{complex-James-periodicity-formula}).

Let $\K_{\Gamma} := \E_{\Gamma}/\mathfrak{m}$ denote the associated Morava
$\mr{K}$-theory.  As an application of \Cref{main2}, we prove: 
\begin{main}
\label{main-tate}
Let $\Smr^1$ act trivially on $\EO_{\Gamma}$. There is a $\K_{\Gamma}$-local equivalence
\[
    \EO_{\Gamma}^{\tSo} \simeq \prod_{-\infty <k<\infty}  \E_{\Gamma}.
\]
\end{main}

\begin{rmk} \label{KOsplit}
\Cref{main2} is a generalization of the splitting \cite[\textsection 15]{GMTate}
\begin{equation}
\label{KOsmCPinfty}
\KO \sm \CPinfty_+ \simeq \KO\vee \bigvee_{k \geq 1} \Sigma^{4k-2} \KU
\end{equation}
and \Cref{main-tate} is a generalization of the fact that
\[ \KO^{\mr{tS^1}} \simeq  \prod_{-\infty <k<\infty} \KU.  \]
\end{rmk}

\begin{rmk}
\label{Hood}
The second author studied $\EO_{\Gamma}$-orientations when $\Gamma$ has height
$p-1$ in his previous work. He proved \cite[Corollary 1.6]{Chatham} that in this
case
\[\orderB{\EO_{\Gamma}, \upgamma} =p.\]
Thus, when $n=p-1$ our bound is not sharp.
\end{rmk}

\begin{rmk} \label{LKW}     
At $p=2$, Kitchloo, Lorman and Wilson \cite{KitchlooLormanWilson,KitchlooWilson}
studied a similar problem. There is a $\Ctw$-action on height $n$ Johnson-Wilson
theory $\E(n)$. The $\Ctw$ fixed points are commonly called real Johnson-Wilson
theory, denoted $\ER(n)$. In \cite{KitchlooLormanWilson,KitchlooWilson}, the
authors use genuine $\Ctw$-equivariant homotopy theory to deduce that 
$\orderB{\ER(n), \upgamma}=2^{n}$. Hahn and Shi \cite{HahnShi} proved that there is
a homotopy ring map $\ER(n)\to \EO_{\Gamma}$, where $\Gamma$ is any
height $n$ formal group over a perfect field of characteristic $2$. Combining
these facts one concludes that $\orderB{\EO_{\Gamma}, \upgamma}$ divides $2^n$  when $\Gamma$
has height $n$. Thus, when $p=2$ our bound on $\orderB{\EO_{\Gamma}, \upgamma}$ is
not sharp.
\end{rmk}
We conjecture:
\begin{conjecture}
\label{order-of-EO-conjecture} If the height of the formal group $\Gamma$ is $n=(p-1)k$, then
$\orderB{\EO_{\Gamma}, \CC } = p^k.$
\end{conjecture} 

A  commutative ring spectrum  $\R$ is called complex orientable if $\upgamma$ is
$\mr{R}$-orientable. In our language, $\R$ is complex orientable if and only if
$\Theta(\R, \upgamma) =1$. Quillen proved  that an $\R$-orientation of
$\upgamma$ is equivalent data to a homotopy ring map $\MU\to \R$ (see \cite[Part II, Lemma~4.6]{Adams}
and \cite[Lemma 4.1.13]{Ravenel} for a statement). Let $\MU[n]$
be the Thom spectrum of the map
\[ \varphi_n\colon \BU\to\BU \]  
given by multiplication by $n$ in the additive $\Einfty$-structure.  If $\R$ is
a ring spectrum $\R$ such that $\Theta(\R, \CC)$ divides $n$ 
we prove that there is a map $\MU[n] \to \R$ using a theorem of
Segal \cite{Segal} (see \Cref{htpy-ring-complex-n-orientation}).  \Cref{main1} and
\Cref{htpy-ring-complex-n-orientation} imply the following corollary.
\begin{cor} \label{cor:applyQuillen} 
There exists a map $\MU[n] \to \EO_{\Gamma}$ whenever $p^{p^k -1}$
divides $n$.
\end{cor}

\subsection*{Acknowledgements}
The authors have benefitted from conversations with Robert Berklund, Pavel
Etingof, Paul Goerss, Jeremy Hahn, Nick Kuhn, Vitaly Lorman, Mike Mandell,
Haynes Miller, Eric Peterson, JD Quigley, Chris Ryba, Andy Senger, Siddharth
Venkatesh and Zhouli Xu.

\section{A brief review of orientation theory}
\label{sec:background}
The goal of this section is to prove a  splitting principle
(see \Cref{splitting-principal}) and generalize a result of Quillen (see
\Cref{htpy-ring-complex-n-orientation}). We do so after reviewing orientation
theory and its connection to chromatic homotopy theory.

Although many of the statements in this
section are known to be true more generally for homotopy commutative ring spectra, we restrict to ones with  an $\Einfty$ structure  
because it is a requirement in the proof of \Cref{htpy-ring-complex-n-orientation} and
\Cref{splitting-principal}.
This is acceptable for our application, namely \Cref{cor:applyQuillen}, because
$\EO_{\Gamma}$ is an $\Einfty$-ring spectrum.
\subsection{Orientation with respect to cohomology theories}
Given an  $\RR$-vector bundle $\upxi$ over a base space $\Bmr_{\upxi}$ with total space
$\Emr_\upxi$, the Thom space of $\xi$ is defined as the cofiber of the projection map
\begin{equation} \label{Thomspace}
\Th(\upxi) \coloneqq \cofiber((\Emr^{\times}_\upxi)_+ \to \Bmr_{\upxi+}),
\end{equation}
where $\Emr^{\times}_{\upxi}$ denotes the complement of the zero section. The
\emph{Thom spectrum} of $\upxi$ is the suspension spectrum of $\Th(\upxi)$:
\[ \Mmr(\upxi) := \Sigma^{\infty} \Th(\upxi)\]
\begin{ex} \label{ex:Thomtriv}
When $\upxi \cong \upepsilon^r$ is a trivial bundle of rank $r$ then
\[ \Th(\upxi) \simeq \susp^{r}\Bmr_{\upxi+}, \]
where $\Bmr_{\upxi+}$ is the base space of $\upxi$.
\end{ex}
\begin{ex}  \label{thomCP} For the complex tautological line
bundle $\upgamma^{n}$ over $\CP^n$,
$\Th\!\left(k\upgamma^n\right)$ is equivalent to
\[ \CP^{n+k}_{k} \coloneqq \cofiber( \CP^{k-1}_+ \to \CP^{n+k}_+), \]
where $k$ is a positive integer and $1 \leq n \leq \infty$. 
\end{ex}
The notion of Thom space does not extend to virtual vector bundles, but the
notion of Thom spectrum does. We recall in brief the construction of Thom
spectrum, the Thom isomorphism, and orientation theory for $\Einfty$ ring
spectra following \cite{ABGHR1}. Many of these ideas originated in
\cite{MayEinfty}. 

The space of units $\GL(\R)$ of an $\Einfty$ ring spectrum $\R$ is an $\Einfty$ space.
Consequently, for any space $\mr{B}$, the set of homotopy classes of pointed
maps $[\mr{B}_+, \BGL(\R)]$ forms an abelian group. 
 
For any zero dimensional virtual vector bundle  $\upxi$, the composite 
\[\begin{tikzcd} \Bmr_{\upxi} \rar["f_{\upxi}"]& \BO \rar["\Jmr"] & \BGL(\Sp)
\rar["\Bmr(\iota_{\R})"] & \BGL(\R), \end{tikzcd}\] 
 classifies a principal $\GL(\R)$-bundle $\Pmr(\upxi, \R)$
which is a right $\GL(\R)$-space where $f_{\upxi}$ is the
classifying map, $\mr{J}$ is the $\mr{J}$-homomorphism and $\iota_\R$ is the
unit map of $\mr{R}$. The $\R$-Thom spectrum of $\upxi$ is defined
as the derived smash product
\[ \M(\upxi, \R) \coloneqq \Pmr(\upxi, \R)_+ \underset{ \GL(\R)_+}{\sm} \R. \]
The above construction of the $\R$-Thom spectrum extends to virtual vector
bundles of all dimensions.
\begin{notation} 
For a virtual vector bundle $\upxi$ of dimension $d$, let $\upxi_0 := \upxi -
\upepsilon^d$ denote the corresponding zero-dimensional bundle. We declare  \[
\M(\upxi, \R) := \Sigma^{d}\M(\upxi_0, \R).\] 
We use $\M(\upxi)$ as a shorthand for $\M(\upxi, \Sp)$. 
\end{notation}
There is a natural weak equivalence
\[ 
\begin{tikzcd}
 \M(\upxi) \sma \R \rar["\simeq"] &\M(\upxi, \R).
\end{tikzcd} \]
which determines the homotopy type of $\M(\upxi, \R)$. 
\begin{notation}  
For all $k \in \ZZ$, let $\CP^{n+k}_k$  denote the Thom spectrum 
\[ \CP^{n+k}_k := \M\!\left( k\upgamma^n\right).\] 
This extends \Cref{thomCP} to all integers, provided that we consider $\CP_a^b$ to be a spectrum.
\end{notation}
An \emph{$\R$-orientation} of $\upxi$ is a choice of trivialization of
$\Pmr(\upxi_0,\R)$, or equivalently a choice of null-homotopy of the classifying map 
\[ 
\begin{tikzcd}
\Bmr(\iota_\R) \circ \Jmr \circ f_{\upxi_0} \colon \Bmr_{\upxi} \rar & \BGL(\R).
\end{tikzcd}\] 
Such a null-homotopy leads to an equivalence of $\R$-modules
\begin{equation} \label[equation-b]{thom-iso-i}
\M(\upxi) \sm \R \simeq \Sigma^d \M(\upxi_0,\R) \simeq \Sigma^d\Bmr_{\upxi+} \sm \R,
\end{equation}
 which we refer to as the \emph{$\R$-Thom isomorphism}.

Dually, an $\R$-orientation for $\upxi$ leads to an equivalence of $\R^{\Bmr_{\upxi+}}$-modules
\[ \R^{\Bmr_{\upxi+}} \simeq \Sigma^d\R^{\M(\upxi)}\] 
which can
be regarded as a cohomological version of the $\R$-Thom
isomorphism. On $\pi_0$, the above isomorphism sends a map $f\colon
\suspinfty\Bmr_{\upxi+} \to \R$ to the composite
\begin{equation} \label{thom-iso-ii}
\begin{tikzcd}
\M(\upxi) \ar[rr,"\mathbbm{1}_{\M(\upxi)} \sma \iota_\R"] && \M(\upxi) \sm\R \rar["\simeq"] & \Sigma^d\Bmr_{\upxi+} \sma \R \rar["f \sm\mathbbm{1}_\R"] &[5pt] \Sigma^d \R \sm \R \rar["\mu_{\R}"] & \Sigma^d\R.
\end{tikzcd}
\end{equation}
When $f\colon \suspinfty\Bmr_+\to \SS\to \R$ is the composite of the collapse
map with the unit of $\R$, we refer to the corresponding element of
$\R^d(\M(\upxi))$ as an \emph{$\R$-Thom class}. It is classical fact that an
$\R$-orientation is equivalent to a choice of an $\R$-Thom class. 
\begin{rmk} 
\label{thom-skeleton} 
If $\upxi$ is $\R$-orientable, then its restriction $\upxi^{(k)}$ to the $k$-th skeleton $\Bmr^{(k)}_\upxi$ is also $\R$-orientable and the Thom
spectrum $\Mmr(\upxi^{(k)})$ is the $(k + \dim \upxi)$-th skeleton of
$\Mmr(\upxi)$. Thus, any $\R$-Thom isomorphism $\upomega\colon \R \sma
\Mmr(\upxi) \overset{\simeq}{\longrightarrow} \R \sma \Sigma^d \Bmr_+$ preserves
the Atiyah-Hirzebruch filtration
 \[ \upomega^{(k)}\colon \R \sma \Mmr(\upxi^{(k)}) \overset{\simeq}{\longrightarrow} \R \sma \Sigma^d \Bmr_+^{(k)}. \]
\ \end{rmk}
\begin{defn} \label{defn:orientorder}
For an $\Einfty$ ring spectrum $\R$, the \emph{$\R$-orientation order}
$\Theta(\upxi, \R)$ of a virtual vector bundle $\upxi$ with basespace
$\mr{B}_{\upxi}$ is defined as the order of the class $[\Bmr(\iota) \circ \Jmr
\circ f_{\upxi_0}] \in [\mr{B}_{\upxi+}, \BGL(\R)].$ 
\end{defn}
\begin{rmk}  \label{rmk:geointerpret}
If $\Theta(\upxi, \R) =n$, then $n$ is the smallest positive integer for which
$n\upxi$ is $\R$-orientable. In particular,
$\upxi$ is $\R$-orientable when  $\Theta( \upxi, \R) =1$.
\end{rmk}

\begin{ex} 
Since $\FF_2^{\times}$ is the trivial group, every vector bundle is
$\HF_2$-orientable. However for $p>2$, not every bundle is  $\HF_p$-orientable.
For example, $\HF_p$-orientation order of the tautological line bundle over
$\mathbb{RP}^\infty$ is  $2$ (as its $\HZ$-orientation order is $2$).
\end{ex}
If a virtual vector bundle $\upxi$ is $\R$-orientable and there is a ring map
$\R\to\mr{T}$, then $\upxi$  is also $\mr{T}$-orientable. In particular, $\order(\upxi, \mr{T})$ divides $\order(\upxi, \mr{R})$.

\begin{rmk} 
\label{extendorientation} 
If $\R$ is a $p$-local ring spectrum and $\upxi$ is an $\HZ$-orientable virtual
vector bundle over a compact base then $\order(\upxi, \R)$ is a power of $p$.
This is because any $p$-local ring spectrum admits a ring map $\SS_{p} \to \R$
and $\order(\upxi, \SS_{p} )$ is a power of $p$  as the higher homotopy groups
of $\GL(\SS_{p})$ are $p$-torsion.  
\end{rmk}
\begin{ex} 
\label{complex-James-periodicity-formula} 
The $\SS$-orientation order of the tautological line bundle $\upgamma^{n}$ over
$\CP^n$ is known for all $n \in \NN$ due to work of  Atiyah and Todd
\cite{AtiyahTodd} (upper bound) and  Adams and Walker \cite{AdamsWalker} (lower
bound). It is given by the formula 
\[
\nu_p(\order( \upgamma^n, \SS)) =
\begin{cases}
\max\left\{r+\nu_p(r) \hskip0.3em : \hskip0.3em 1\leq r\leq\left\lfloor\frac{n}{p-1}\right\rfloor \right\}& \text{if $p\leq n+1$}\\
0 & \text{if $p>n+1$,}
\end{cases}\]
where $\nu_p$ is the $p$-adic valuation. 
\end{ex}

\subsection{Complex $n$-orientation and the splitting principle}
A \emph{complex orientation} of a  $\Einfty$-ring spectrum $\R$ is a choice of
an $\R$-Thom class for the tautological complex line bundle $\upgamma$, i.e. a map
\[ u_{\R}\colon \M(\upgamma-\upepsilon^2) \simeq \susp^{-2}\CPinfty\to \R \] 
such that the diagram 
\[\begin{tikzcd}
\Sp\dar \rar["\iota_{\R}"] & \R\\
 \susp^{-2}\CPinfty\urar["u_{\R}"', end anchor = -140]
\end{tikzcd}\]
commutes up to homotopy. One may generalize this definition to the notion of a
complex $n$-orientation as follows. 
\begin{defn}
\label{defn:classical-complex-orientation}
A \emph{complex $n$-orientation} of an $\Einfty$-ring spectrum $\R$ is a map
 \[ u_{\R}\colon \M(n(\upgamma -\upepsilon^2)) \simeq \susp^{-2n}\CPinfty_n\to \R \]
such that the diagram 
\[\begin{tikzcd}
\Sp\dar \rar["\iota_{\R}"] & \R\\
\susp^{-2n}\CPinfty_n\urar["u_{\R}"', end anchor = -140]
\end{tikzcd}\]
commutes up to homotopy.
\end{defn}

Quillen proved that complex orientations of a ring spectrum $\R$ are in
bijection with homotopy ring maps $u_1\colon \MU\to \R$ (see \cite[Lemma
4.1.13]{Ravenel} for a proof). Stated differently, Quillen proved that if $\order(\upgamma, \R) =1$ then $\order(\upxi, \R) =1$
for any complex vector bundle $\upxi$.

Let $\MUn$ be the Thom spectrum of the multiplication by $n$ map
$\varphi_n\colon \BU \to \BU$. Because $\varphi_n$ is an $\Einfty$-map (with
respect to the additive infinite loop structure on $\BU$), $\MUn$ is
an $\Einfty$-ring spectrum. We make the following generalization of Quillen's
work. 
\begin{thm}
\label{htpy-ring-complex-n-orientation}
Let $\R$ be an $\Einfty$-ring spectrum. There is a one-to-one correspondence between complex $n$-orientations of $\R$ and 
homotopy classes of unital maps $u_n\colon \MUn\to \R$.
\end{thm}
\begin{proof}
A complex $n$-orientation of $\R$ is a null-homotopy of
the composite 
\[ 
\begin{tikzcd}
\CPinfty \rar["f_{\upgamma}"] & \BU \rar["\varphi_n"] & \BU \rar["\mr{J}_{\CC}"] & \BGL(\SS) \ar[rr,"\BGL(\iota_{\R})"] && \BGL(\R)
\end{tikzcd}
\]
where $\mr{J}_\CC$ is the complex $\mr{J}$-homomorphism.  Using the fact that $\BU$, $\BGL(\SS)$ and  $\BGL(\R)$ are
$\Einfty$-spaces and $\varphi_n$, $\mr{J}_\CC$ and $\BGL(\iota_{\R})$ are
$\Einfty$-maps, we can extend $f_{\upgamma}$ to $\tilde{f}_{\upgamma}$
\[ 
\begin{tikzcd}
\CPinfty \dar[hook, "\iota"]  \rar["f_{\upgamma}"] & \BU \ar[rrr, "\BGL(\iota_{\R})\circ \mr{J}_\CC \circ \varphi_n "] &&& \BGL(\R) \\
\Qmr\CPinfty \ar[ur, "\tilde{f}_{\upgamma}"']
\end{tikzcd}
\]
and the null-homotopy of $\BGL(\iota_{\R})\circ \Jmr_\CC \circ \phi_n \circ
f_{\upgamma}$ to a null-homotopy of \[ \BGL(\iota_{\R})\circ \Jmr_\CC \circ
\varphi_n \circ \tilde{f}_{\upgamma} \colon \Qmr\CPinfty \to \BGL(\R).\] By
\cite[Proposition 2.1]{Segal} there exists a map $\alpha\colon \BU \to \Qmr\CPinfty$  such
that $\tilde{f}_{\upgamma} \circ \alpha \simeq \mathbbm{1}_\BU$. Therefore, 
\[ \BGL(\iota_{\R})\circ \Jmr_\CC \circ \phi_n \simeq \BGL(\iota_{\R})\circ
\Jmr_\CC \circ \phi_n \circ \tilde{f}_{\upgamma} \circ \alpha \simeq \ast.
\] 
 We let
$u_n\colon\MUn \to \R$ to be the corresponding
$\R$-Thom class. 
\end{proof}
\begin{rmk}
\label{rmk:htpy-ring-complex-n-orientation}
\Cref{htpy-ring-complex-n-orientation} restricted to $n=1$ is significantly
weaker when compared to Quillen's result (see \cite[Lemma 4.1.13]{Ravenel}).
Quillen only requires that the ring $\R$ to be homotopy commutative and he
proves that the $\R$-Thom class $u_1$ is muliplicative. Quillen's argument when
$n=1$ relies on the computation of $\R^*(\BU(d)_+)$ in terms of Chern classes to
gain control over the multiplicative structure. It is unclear how to generalize such direct methods to $n>1$. 
\end{rmk}
\begin{rmk}
\label{phi-o-f-classifies-nV}
If $f_{\upxi}\colon \mr{X} \to \BU$ is the classifying map of a virtual complex vector bundle
$\upxi$, then $\varphi_n \circ f_{\upxi}$ classifies the bundle $n \upxi$.
\end{rmk}
\begin{lem}[The splitting principle]
\label{splitting-principal}
If $\R$ is an $\Einfty$-ring spectrum then
$\orderB{\R, \upxi}$  divides  $\orderB{\R,\upgamma}$
for any virtual complex vector bundle $\upxi$.
\end{lem}
\begin{proof}
Without loss of generality, we may assume $\dim \upxi =0$.  By \Cref{phi-o-f-classifies-nV}, the classifying map of $\upxi^{\oplus n}$
factors through $\varphi_n$ 
\[ 
\begin{tikzcd}
f_{n \upxi}\colon \Bmr_{\upxi+} \rar["f_{\upxi}"] & \BU \rar["\varphi_n"] & \BU. 
\end{tikzcd}
\]
 In the proof of \Cref{htpy-ring-complex-n-orientation}, we show that if $n\upgamma$ is $\R$-orientable then the composite 
 \[ 
\begin{tikzcd}
\BU \ar[rrr, "\BGL(\iota_{\R})\circ \mr{J}_\CC \circ \varphi_n "] &&& \BGL(\R) 
\end{tikzcd}
\]
is null-homotopic. Therefore the composite $\BGL(\iota_{\R})\circ \mr{J}_\CC \circ f_{n\upxi}$ is null as well
\[ 
\BGL(\iota_{\R})\circ \mr{J}_\CC \circ f_{n\upxi}  = (\BGL(\iota_{\R})\circ \mr{J}_\CC \circ \varphi_n) \circ f_{\upxi}  \simeq (\ast) \circ f_{\upxi} \simeq \ast.
\]
Thus, when $\orderB{\R,\upgamma} = n$ then $n \upxi$ is $\R$-orientable,  and therefore, $\orderB{\R,\upxi}$ divides $n$.  
\end{proof} 

\section{The \texorpdfstring{$\Cp$}{Cp}-action on the Morava \texorpdfstring{$\K$}{K}-theory of \texorpdfstring{$\CP^{\infty}$}{CP\infty}  }
\label{sec:Cp-action}
Let $\Gamma$ be a formal group of height $n = (p-1)k$ for some positive $k$ such
that $\mr{C}_p$ is a subgroup of  $\Aut(\Gamma)$ (see \Cref{heightrestrict}).
The goal of this section is to study the action of $\mr{C}_p$ on $\K_{\Gamma}^{\
*} \CP^\infty$, where $\K_{\Gamma} = \E_{\Gamma}/\mathfrak{m}$ is the associated
Morava $\K$-theory.  We do so by  relating the action of  $\Cp$ with the action
of an element $\Pmr_k$ of the Steenrod algebra on the homology of $\CP^\infty$
(see \Cref{compare-chibar-Pk}).

The ring spectrum $\EO_{\Gamma} := \E_{\Gamma}^{\mr{h}\Cp}$ 
depends not only on the formal group $\Gamma$ but also on the choice of
embedding $\iota\colon \Cp \hookrightarrow \Aut(\Gamma)$ up to
a conjugation. We emphasize that our results apply
to all choices of $\EO_{\Gamma}$ because
\Cref{tk-valuation}  holds for all pairs $(\Gamma,
\iota)$. We begin by recalling some facts from Lubin-Tate theory needed in this
paper.

Let $\FFtwee$ denote the separable closure of $\FF$ and let $\Gammatwee$ be the
base change of $\Gamma$ to $\FFtwee$. The endomorphism ring of $\Gammatwee$
 is a noncommutative valuation ring which can be explicitly described as 
\begin{equation}
\label{End-Gammatwee-identification}
\End(\Gammatwee)\cong \WW(\FFtwee)\langle \Tmr\rangle /(\Tmr a - \phi(a)\Tmr, \Tmr^n - p{\sf u}),
\end{equation}
where $\Tmr$ is a uniformizer, $\phi$ is the Frobenius and
${\sf u}\in \WW(\FF_{p^n})^{\times}$ is a unit (see \cite{LubinTate-FormalModuli}). Note that $v(\Tmr^n) = v(p{\sf
u}) = 1$. Any element $e \in \End(\Gammatwee)$ can be expressed as 
\begin{equation} \label{ps-expansion}
 e = \sum_{i=0}^\infty a_i \Tmr^i
\end{equation}
where $a_i$ are Teichm\"uller lifts of $\FFtwee$ in $\WW(\FFtwee)$
and the valuation is $v(e)=j/n$, where $j$ is the smallest integer such that $a_j\neq 0$.

\begin{notation} \label{notation-for-rest}     
We will use the following notations and conventions for the remainder of the
paper.
\begin{itemize}
\item Fix a prime $p$ and a positive  integer $k$. Let  $n = k (p-1)$.

\item 
    Fix a perfect field $\FF$ of characteristic $p$ and a formal group
    $\Gamma$ of height $n$ over $\FF$ such that $\Cp$ is a subgroup of $\Aut(\Gamma)$. Let $\E_{\Gamma}$ denote the associated
    Lubin-Tate theory. By Lubin-Tate theory \cite{LubinTate-FormalModuli} 
    \[ \pi_*\E_\Gamma \cong \WW(\FF)\llbracket u_1, \dots, u_{n-1}\rrbracket[u^{\pm}], \]
    where $\WW(\FF)$ is the ring of Witt vectors over $\FF$.
    The elements $u_1$, \ldots, $u_{n-1}$ are elements of $\pi_0\E_{\Gamma}$ and $u\in\pi_{-2}\E_{\Gamma}$.

\item
    Let $\mfrak$ denote the maximal ideal $(p, u_1,\dots, u_{n-1})$ of
    $\pi_*\E_{\Gamma}$ and let $\K_{\Gamma}$ denote the corresponding height $n$
    Morava $\K$-theory so that $\pi_*\K_{\Gamma} \simeq \pi_*(\E_{\Gamma})/\mfrak.$
\item 
    Fix an embedding $\iota\colon \Cp\to \Aut(\Gamma)$  and let
    $\EO_{\Gamma} \coloneqq \E_{\Gamma}^{\hCp}.$
\item 
    Abbreviate $\E_{\Gamma}$ by $\E$, $\K_{\Gamma}$ by $\K$, and $\EO_{\Gamma}$
    by $\EO$, leaving the dependence on $\Gamma$ and $\iota$ implicit.
\item 
    Let $\Ects_*(-) \coloneqq \pi_* (\Lmr_{\K}(\E \sm -))$, $\EEO_*(-) \coloneqq
    \pi_*(\E \sma_{\EO} - )$, and  $\KEO_*(-) = \pi_*(\K\sm_{\EO}-)$.
\item 
    Fix a $p$-typical coordinate on $\Gamma$ and let $\pi_{\E}\colon \BP \to \E$
    be the associated map. Let $\pi_\K\colon \BP \to \K$ denote the composite of
    $\pi_{\E}$ with the reduction map from $\E \to \K$.

\item Let $\pi_{\Fp}\colon \BP \to \HFp$ be the standard reduction map.
\end{itemize}
\end{notation}

\bigskip

The map $\pi_{\E}\colon \BP\to\E$ induces a map $\BP\sm\BP\to\E\sm \E \to \LK(\E\sm \E)$.
There is a homotopy coequalizer map $\LK(\E\sm\E)\to\E\sm_{\EO}\E$.
These maps fit into a diagram
\begin{equation}
\label[diagram]{diagram:BP_*BP-to-maps}
\begin{tikzcd}[
 row sep=small,
 cong/.style={"\cong" {sloped, description, yshift=0pt,#1}, phantom}
]
\BP_*\BP \rar["\rho_{\E}"] & \Ects_*\E\dar[cong] \rar & \EEO_*\E\dar[cong] \\
              & \MapsCts(\Aut(\Gamma), \E_*) \rar["\Res^{\Cp}"]& \Map(\Cp, \E_*).
\end{tikzcd}
\end{equation}
The vertical isomorphisms are by Galois theory \cite[Theorem 5.4.4 and
Definition 4.1.3]{Rognes-galois}. The map $\Res^{\Cp}$ is restriction along the
inclusion map $\iota\colon \Cp\to\Aut(\Gamma)$. 

\begin{notation}
\label{notation:theta(g)}
Given $\uptheta\in\BP_*\BP$ and ${\sf g}\in\Aut(\Gamma)$,
the map
\[
\begin{tikzcd}
\rho_{\E}\colon \BP_*\BP \rar &  \Ects_*\E=\Map^c(\Aut(\Gamma),\E_*)
\end{tikzcd}
\]
allows us to interpret the element $\uptheta({\sf g})\coloneqq\rho_{\E}(\uptheta)({\sf g})\in \E_*$.
Let $\overline{\uptheta}({\sf g})$ denote the image of $\uptheta({\sf g})$ under the quotient map
\[ \E_* \twoheadrightarrow \E_*/ \mfrak \cong \K_*.\]
\end{notation}

There is an isomorphism $\BP_*\BP\cong \BP_*[t_1,t_2,\ldots]$. 
\begin{lem}
\label{tk-valuation}
Let $\upzeta\in \Aut(\Gamma)$ be an element of order $p$.
Then $\tbar_i(\upzeta)=0$ for $i<k$ and $\tbar_k(\upzeta)$ is a unit.
\end{lem}

\begin{proof}
\def\Ktwee{\widetilde{\Kmr}} 
Let $\FFtwee$ be the separable closure of $\FF$, 
$\Gammatwee$ be the base change of $\Gamma$ to $\FFtwee$ and $\Ktwee
:=\K_{\Gammatwee}$ be the Morava $\K$-theory associated to $\Gammatwee$.
The field extension $\FF \to \FFtwee$ induces the inclusion of groups 
\[ 
\begin{tikzcd}
\alpha \colon\Aut(\Gamma) \rar[hook] & \Aut(\Gammatwee)
\end{tikzcd}
\] 
as well as a map 
\[
\begin{tikzcd} 
f\colon \K_* \rar & \Ktwee_*
\end{tikzcd}
\]
on Morava $\K$-theories. Since $f$ is an injection and 
\[\tbar_i(\alpha(\upzeta))=f(\tbar_i(\upzeta)),\] 
it suffices show that $\tbar_i(\tilde{\upzeta})=0$ for $i<k$ and
$\tbar_k(\tilde{\upzeta})$ is a unit whenever
$\tilde{\upzeta}\in\Aut(\Gammatwee)$ is an element of order $p$. When expressed
as a power series (as in \eqref{ps-expansion})
\[ \tilde{\upzeta} =  1 + \sum_{i>0} a_i\Tmr^i,  \] 
and it follows  from  the properties of $t_i$ that 
\begin{equation}
\label{ai}
 t_i(\tilde{\upzeta}) \equiv a_i u^{1-p^i} \mod \mathfrak{m} 
\end{equation}
(see \cite[2-6]{K2Z} or \cite{Ravoddarf}[pg. 437]). Because $\Qp(\tilde{\upzeta})$ is a totally ramified
extension of $\Qp$ of degree $p-1$ with $\tilde{\upzeta}-1$ as a uniformizer,
\[ v(\tilde{\upzeta} -1) = \frac{1}{p-1} = v(\Tmr^{k}).\] 
Thus, $a_i = 0$ if $i < k$ in \eqref{ai} and a unit if $i= k$.
\end{proof}
\subsection{Filtering the $\Cp$-action on $\K_*\Xmr$}

For any spectrum $\Xmr$ the $\EEO_*\E$-coaction on $\E_*\Xmr$, given by the
composition
\begin{equation} \label{EEOaction}
\begin{tikzcd}
\Psi \colon \E_*\Xmr \rar[" \Psi_\E"] &
\Ects_*\E \otimes_{\E_*} \E_*\Xmr \rar["\Res^{\Cp}\otimes \id"] &[20pt]
\EEO_*\E \otimes_{\E_*} \E_*\Xmr,
\end{tikzcd}
\end{equation}
 leads to a contragradient action of $\Cp$ as $\EEO_*\E\cong \Map(\Cp ,\E_*)$. 
Explicitly, the action of ${\sf g} \in \Cp$ is given by the map 
\begin{equation} \label{eqn:action}
\begin{tikzcd}
\E_*\Xmr \rar["\Psi"] &
\EEO_*\E\otimes_{\E_*} \E_*\Xmr \rar["\ev_{{\sf g}}\otimes 1"] &
\E_*\Xmr.
\end{tikzcd}
\end{equation}
Note that $\EEO_*\E\cong \Map(\Cp,\E_*)$ is a quotient Hopf algebra of
$\Ects_*\E$ and dual to the Hopf algebra $\pi_*(\EOmod(\E,\E)) \cong \ECpsigma$ as
defined in \Cref{twist}.
\begin{rmk}
\label{twist}
The nontrivial action of $\Cp$ on $\E_*$ means  $ \pi_*(\EOmod(\E,\E))$  is isomorphic as a ring to the \emph{twisted} group ring
\[\ECpsigma\coloneqq \E_*\langle \upzeta \rangle/(\upzeta^p=1,  \upzeta \cdot e
= \upzeta(e) \cdot \upzeta),\] 
where $e \in \E_*$. Since the action of $\Cp$ on $\K_*$ is trivial,
$\ECpsigma/\mfrak$ is isomorphic to the untwisted group ring $\KCp$.
\end{rmk}
\begin{defn}
\label{defn:even}
A spectrum $\Xmr$ is even if $\Xmr$ is bounded below, $\HZ_{2i+1}\Xmr=0$ and
$\HZ_{2i}\Xmr$ is a finitely generated $\ZZ$-module for all integers $i$.
\end{defn}
When $\Xmr$ is an even spectrum, there is an isomorphism
 \[ \K_* \Xmr \cong \E_* \Xmr/\mfrak \] which makes $\K_*\Xmr$ into a
$\KCp$-module. 

Next, we use the Atiyah-Hirzebruch filtration to relate the
$\KCp$-module structure  on $\K_*\Xmr$ to the $\mathcal{B}(k)_*$ comodule
structure on $\H_*\Xmr$ for an even spectrum $\Xmr$, where $\mathcal{B}(k)_*$ is a
quotient Hopf algebra of the even dual Steenrod algebra. The Atiyah-Hirzebruch
filtration on $\R_*\Xmr$ is an increasing filtration induced by the skeletal
filtration of $\Xmr$
\[\Fil_d\R_*\Xmr :=\R_*\Xmr^{(d)}\subseteq \R_*\Xmr.\]
The associated graded
\[ \gr \R_*\Xmr \coloneqq \bigoplus \frac{\Fil_{d}\R_*\Xmr}{\Fil_{d-1}\R_*\Xmr}\]
is bigraded as the second grading is induced by the internal grading of $\R_*\Xmr$.

When $\R$ is complex orientable (e.g. $\BP$, $\K$ and $\E$) and $\Xmr$ is an
even spectrum (e.g $\CP^{n+k}_n$), the AHSS (Atiyah Hirzebruch spectral
sequence) for $\R_*\Xmr$ collapses on the $\E_2$-page. This collapse induces an
isomorphism
\[
\gr \R_*\Xmr \cong \R_*\otimes \H_* \Xmr.
\]
\begin{lem}
\label{lem:filtK}
Let $\Xmr$ be an even spectrum and let $\chi = \upzeta - 1 \in \KCp$.
Then
\[\chi_*\Fil_d\K_*\Xmr\subseteq \Fil_{d-2p^{k}+2}\K_*\Xmr.\]
\end{lem}
\begin{proof}
We required that $\Xmr$ be even so the map $\K_*\otimes_{\BP_*}\BP_*\Xmr\to \K_*\Xmr$
is surjective and we can check our claim on the image.

Pick $x^{\BP}\in \Fil_d\BP_* \Xmr$.
Write $x^{\K} = \pi_{\K*}(x^{\BP})$ and $x^{\Fp}=\pi_{\Fp*}(x^{\BP})$.
Write the $\BP$-coaction on $x^{\BP}$ as
\[\Psi(x^{\BP})=1\otimes x^{\BP} + \sum \uptheta_{(1)} \otimes x_{(2)}^{\BP}.\]
The counit axiom says that
\[(\epsilon\otimes 1)(\Psi(x^{\BP}))=x^{\BP}=(\epsilon\otimes 1)(1\otimes x^{\BP})\]
so we may assume that $\uptheta_{(1)}\in\ker\epsilon = (t_1,t_2,\ldots)$.
By definition,
\[ \chi_*(x^\K) = (\upzeta-1)_*(x^{\K}) = \sum
\overline{\uptheta}_{(1)}(\upzeta) \cdot x_{(2)}^{\K}\] 
where $\overline{\uptheta}_{(1)}$ is as in \Cref{notation:theta(g)}. By
\Cref{tk-valuation}, $\overline{\uptheta}_{(1)} (\upzeta) =0$ when 
\[ \overline{\uptheta}_{(1)} \in (t_1,\ldots, t_{k-1}) \subseteq \BP_*\BP, \]
therefore $\chi_*(x^\K) \in  \Fil_{d-|t_k|} \K_*\Xmr  =
\Fil_{d-2p^{k}+2}\K_*\Xmr$ as $t_k$ is the element of lowest degree in $\ker
\epsilon/(t_1, \dots, t_{k-1})$.
\end{proof}

 We introduce an increasing
filtration on $\K_*[\Cp]$ by  assigning
$\chi$ an  `Atiyah Hirzebruch weight'
\[ |\chi|_{\mr{AH}} = -|t_k| = 2- 2 p^k.\]  
We denote the associated graded by $\grKCp$ and the representative of $\chi$ in
$\grKCp$  by $\chibar$.  It follows from \Cref{lem:filtK} that, for an even
spectrum $\Xmr$, the $\K_*[\Cp]$-module structure on $\K_*\Xmr$ induces a
$\grKCp$-module structure on $\gr \K_*\Xmr$ (see \Cref{leading-term-of-grK(X)}).
\begin{lem} \label{graded-Cp}
The bigraded Hopf algebra $\grKCp$ is isomorphic to
\[ \K_*[\chibar]/(\chibar^p), \]
where $\chibar$ is a primitive in the Atiyah Hirzebruch filtration $2 -2p^k$.
\end{lem}
\begin{proof}
Since $\Delta(\upzeta) = \upzeta \otimes \upzeta$ and $\chi=\upzeta-1$,  we see $\Delta(\chi) = \chi \otimes 1 + 1 \otimes \chi + \chi\otimes\chi.$ 
Thus, in the associated graded $\Delta(\chibar) = \chibar \otimes 1 + 1 \otimes \chibar$. 
\end{proof}

 Let $\Pcal$ be the quotient Hopf algebra of the Steenrod algebra
\[\Pcal=\Acal \mmod (\Qmr_0,\Qmr_1,\ldots)\] 
where $\Qmr_i$ are the Milnor primitives. 
Its dual is the sub Hopf
algebra of $\Acal_*$ generated by the even degree generators 
\[
\Pcal_* \cong
\begin{cases}
\Fp[\xi_1, \xi_2, \dots ] \subset \Acal_* & \text{if $p$ is odd} \\
\Fp[\xi_1^2, \xi_2^2, \dots ] \subset \Acal_* & \text{if $p=2$}.
\end{cases}
\]
Under the map $(\pi_{\FF_p} \sma \pi_{\FF_p})_*\colon \BP_*\BP \to \A_*$ the image of  $t_k$ is
\[(\pi_{\FF_p} \sma \pi_{\FF_p})_*(t_k)=
    \begin{cases}
    c(\xi_k) & \text{if $p$ is odd} \\
    c(\xi_k^2) & \text{if $p=2$,}
    \end{cases}
\]
where $c$ denotes the antipode map of the dual Steenrod algebra. It follows from the definition of $\xi_k \in \Pcal_*$  that its  linear dual  $\Pmr_k \in \Pcal$ satisfies the formula 
\begin{equation} \label{defnPk}
 \Pmr_k(x) = x^{p^k} 
 \end{equation}
where $x \in \H^2\CPinfty_+$ is a generator (see \cite{MIlnorSteenrod}).

\begin{defn} \label{defn:double-steenrod}
Let $\Bcal(k) \subset \Pcal$ denote the sub Hopf algebra generated by $\Pmr_k$.
As a Hopf algebra
\[\Bcal(k)\cong \Fp[\Pmr_k]/(\Pmr_k^{p})\]
where $\Pmr_k$ is a primitive.
\end{defn}
 Note that  $c(\xi_k ) \equiv -\xi_k \mod (\xi_1, \dots, \xi_{k-1})$, which leads to the negative sign in \Cref{leading-term-of-grK(X)}.  Let 
\[ 
\begin{tikzcd}
\upiota_{\Xmr}\colon \H_*\Xmr \rar & \gr \K_* \Xmr \cong \K_* \otimes \H_* \Xmr
\end{tikzcd}
\]
be the map that sends $x \mapsto 1 \otimes x $.
\begin{lem} \label{leading-term-of-grK(X)} For an even spectrum $\Xmr$
\begin{equation} \label{compare-chibar-Pk}
 \chibar (\upiota_\Xmr(x)) = -\tbar_k(\upzeta)\upiota_\Xmr(\Pmr_k(x)) 
 \end{equation}
for all  $x \in \H_* \Xmr$.
\end{lem}
\begin{proof}
Let $p$ be an odd prime. Let  $x^{\Fp}=x$ and $y^{\Fp}=y=\Pmr_{k}(x)$,  $\Ical^{\Fp}_r = \Pcal_*\otimes \Fil_{r}\H_*\Xmr$ and $\Ical^{\BP}_r = \BP_*\BP\otimes_{\BP_*} \Fil_{r}\BP_*\Xmr$.
The $\Acal_*$ coaction on $x^{\Fp}$ is
\begin{align*}
\Psi(x^{\Fp}) &\equiv 1\otimes x^{\Fp}\, + \xi_k\otimes y^{\Fp}\, \mod{ (\xi_1,\ldots,\xi_{k-1}, \Ical^{\Fp}_{d-2p^k}}).\\
\intertext{
    Let $x^{\BP}$ be a lift of $x$ to $\BP_d\Xmr$.
    It follows that the $\BP$ coaction on $x^{\BP}$ is
}
\Psi(x^{\BP}) &\equiv 1\otimes x^{\BP} - t_k\otimes y^{\BP} \mod{(p,\ldots, v_{n-1},t_1,\ldots,t_{k-1}, \Ical^{\BP}_{d-2p^k}}),
\end{align*}
where $y^{\BP}$ is a  lift of $y^{\Fp}$. Since $\overline{\uptheta} (\upzeta) \equiv 0 \mod \mathfrak{m}$ for  $\overline{\uptheta} \in (t_1, \dots, t_{k-1})$, we get 
\[ \upzeta \cdot x^\K \equiv  x^\K - \overline{t}_k(\upzeta) y^\K \mod \Fil_{d - 2p^k}\K_*\Xmr. \]
Since $\upiota_\Xmr( x) = [\pi_{\K*}(x^{\BP})] $ and $\upiota_\Xmr( \Pmr_k(x)) = [\pi_{\K*}(y^{\BP})] $, the result follows.

The same argument works for $p=2$ after replacing $\xi_i$ with $\xi_i^2$ above. 
\end{proof}
\begin{rmk} Following \Cref{graded-Cp} and  \Cref{leading-term-of-grK(X)}, we
see that there is a Hopf algebra isomorphism
\[ \K_* \otimes \mathcal{B}(k) \cong \gr \K_*[\Cp]\] obtained by sending $\Pmr_k
\mapsto  -\overline{t}_k(\upzeta) \chibar$. This isomorphism relates the $\K_*
\otimes \mathcal{B}(k)$-module structure on the left to the $\gr
\K_*[\Cp]$-module structure on the right, of the isomorphism 
\[ \K_* \otimes \H_* \Xmr \cong \gr \K_*\Xmr \]
for any even spectrum $\Xmr$.
\end{rmk}

\subsection{The action of $\Cp$ on $\K_*\CPinfty_+$}
We now explicitly compute the coaction of $\Bcal(k)$ on $\H_*\CPinfty_+$ and
deduce that $\K_*\CPinfty_+$ has a large $\KCp$-free summand.

\begin{notation}
Let $x$ be the generator of $\H^*\CP^\infty_+ \cong \FF_p\llbracket x\rrbracket$
and let $b_i \in \H_{2i}\CP^\infty_+$ be the linear dual of $x^i$. The
$\HFp$-Thom isomorphism for $c \hspace{1pt} \upgamma$ implies $\H^*\CP^\infty_c$ is a free
module of rank one over $\H^*\CPinfty_+$.  Fix an $\HFp$-Thom class $u_c \in
\H^{2c}\CP^\infty_c $.  Let 
\[ b_i \in \H_{2i} \CPinfty_c\] 
 denote the linear dual to $x^{i-n} \cdot u_{c}$. We also use $b_i \in \H_{2i}\CP^{a+ c}_c$ 
to denote the element which maps to $b_i$ under the skeletal
inclusion $\CP^{a+ c}_c \to \CP^{\infty}_c$. 
\end{notation}
\begin{notation} Let $\bk := \frac{|\Pmr_k^{p-1}|}{2} =(p-1)(p^k -1)$.
\end{notation}
\begin{prop}
\label{grK*CPinfty}
Let $c$ be an integer. The action of $\chibar$ on \[ \grK_* \CPinfty_c \cong
\K_*\{b_c,b_{c-1}, \ldots\} \] is given by
\[ \chibar_*(b_i) = \left\lbrace
 \begin{array}{ccc}
 (i-p^k+1)b_{i-p^k+1} & \text{if \ $i \geq p^k -1 +c$  }\\
 0 & \text{otherwise.}
 \end{array} \right.\]
Moreover, there is a $\grKCp$-module isomorphism
\[ \grK_*\CPinfty_c \cong \Fmr_c^{ \ \gr\K} \oplus \Mmr_c^{\ \gr \K}\] 
where $\gr\Fmr_c$ is a free $\grKCp$-module of infinite rank and $\gr\Mmr_c$ is
a finite dimensional $\K_*$-module.
\end{prop}

\begin{rmk}
\label{selfinjective}
The Hopf algebras $\KCp$ and $\grKCp$ are self injective. In other words, if $\mr{M}$
is a $\KCp$-module and \[ i\colon\KCp\{\iota \}\to \mr{M} \] is an injective map then
$i$ admits a splitting (and similarly for $\grKCp$-modules).
\end{rmk}

\begin{proof}
By \Cref{leading-term-of-grK(X)}, it suffices to compute the action of
$\Bcal(k)$ on $\H_*\CPinfty_c$. Since $\Pmr_k \in \Pcal$ is primitive (see
\Cref{prim}), its action on  $\H^*\CPinfty \cong \FF_p[\![x]\!]$ follows the Leibniz rule
\[ \Pmr_k(x^i) = ix^{i-1}\Pmr_k(x).\] 
The action of $\Pmr_k$ on $\H^*\CPinfty_c \cong \FF_p[\![x]\!] \cdot u_{c}$ can
be calculated using the formula
\[ \Pmr_k(x^i \cdot u_c) = \Pmr_k(x^i) \cdot u_{c} + x^i \cdot  \Pmr_k(u_{c}).\]

By \Cref{defnPk}, $\Pmr_k(x) = x^{p^k}$ and 
$\Pmr_k (u_c) = cx^{p^k-1}\cdot u_c$. Therefore,
\[ \Pmr_k(x^i \cdot u_c) = (i+c) x^{i+p^k-1} \cdot u_{c}. \]
Dually 
\[ \Pmr_k(b_i) = (i-p^k+1) b_{i-p^k+1},\] 
where it is understood that $\Pmr_{k}(b_i) = 0$ if $i-p^k+1<c$. Thus, $\gr\K_*
\CPinfty_c$ contains an infinite rank free $\grKCp$-submodule $\Fmr^{\ \gr
\K}_c$ generated by the set
\[\{ b_{pi}\colon i > (\bk - c)/p \}.\]
By \Cref{selfinjective}, $\Fmr^{\ \gr \K}_c$ is a summand of $\gr\K_*
\CPinfty_c$. We denote its complement by $\Mmr_c^{\ \gr\K}$.
\end{proof} 

\begin{rmk}
\label{prim}
If $p=2$, then $\Delta(\Pmr_k) = \Pmr_k\otimes 1 + \Qmr_k\otimes \Qmr_k +
1\otimes \Pmr_k$ in $\Acal$ but the quotient map $\Acal\to \Pcal$ sends
$\Qmr_k\mapsto 0$ and so we see that $\Pmr_k$ is primitive as an element of
$\Pcal$ even though it is not a primitive element of $\Acal$. If $p$ is odd,
$\Pmr_k$ is primitive even in $\Acal$.
\end{rmk}

\begin{lem}
\label{lift-free-summand-from-grKCp}
Let $\Mmr$ be a $\KCp$-module which admits an increasing
filtration 
\[\{ 0 \} = \Fil_{-c}\Mmr  \subset \dots \subset \Fil_0 \Mmr \subset \Fil_1\Mmr \subset \dots \subset \Mmr  \]
so that $\Mmr = \colim_i \Fil_i\Mmr$. Suppose that the action of $\K_*[\Cp]$ on
$\Mmr$ satisfies 
 \[ \chi \Fil_k\Mmr \subset \Fil_{k - 2p^k +2} \Mmr\]
so that $\gr \Mmr$ is a $\gr \K_*[\Cp]$-module. Also, suppose there is an injective
$\grKCp$-module map  
\[ 
\begin{tikzcd}
 \tilde{\upalpha}\colon\tilde{\Fmr} \rar &\gr \Mmr ,
\end{tikzcd}
\]  
where $\tilde{\Fmr}$ is a free $\grKCp$-module.  Then there exists  a lift of
$\tilde{\upalpha}$ to an injective $\KCp$-module map
\[ 
\begin{tikzcd}
\upalpha\colon \Fmr \rar & \Mmr,
\end{tikzcd}
\]
where $\Fmr$ is a free $\KCp$-module. In fact, $\upalpha$ is the inclusion of a
summand.
\end{lem}
\begin{proof}
Let us choose a basis so that we may write 
\[ \tilde{\Fmr}= \bigoplus_{i \in \mr{I}}\grKCp\{\tilde{\iota}_i\} \text{ and }
  \Fmr= \bigoplus_{i \in \mr{I}}\K_*[\Cp]\{\iota_i\}.\] 
By assumption, any element $m_i \in \Mmr$ such that  
$[m_i]= \tilde{\upalpha}(\tilde{\iota}_i)$ satisfies 
\[ \chibar^i [m_i] = [\chi^i(m_i)].\]
 Fix such an $m_i$ for each $i > -c$. Now
define the map $\upalpha$ by setting $\upalpha(\iota_i) = m_i$ and extending it
linearly to a $\K_*[\Cp]$-module map. Clearly $\upalpha$ is injective as
$\tilde{\upalpha}$ is injective, and it splits as $\K_*[\Cp]$ is self-injective
(see \Cref{selfinjective}).
\end{proof}
\Cref{grK*CPinfty} and \Cref{lift-free-summand-from-grKCp} imply:
\begin{cor}
\label{KCPinfty-decomposition}
 For any $c \in \ZZ$, there exists a $\KCp$-module  isomorphism
\[ \K_*\CPinfty_c \cong \Fmr_c^\K \oplus \Mmr_c^\K,\] 
where  $\Fmr_c^\K$ is  free  and $\Mmr_c^\K$ is finitely generated  as $\KCp$-modules.
\end{cor}

\section{An upper bound on the $\EO$-orientation order of $\upgamma$}
\label{sec:main-results}
In \Cref{sub:splitting}, we lift the $\KCp$-module splitting of
\Cref{KCPinfty-decomposition} to an $\EO$-module splitting 
(see \Cref{splittingEOCPc}). In \Cref{sub:EOThom} we show that the inclusion of the
compact summand of $\EO \sma \CPinfty_c$  factors through $\EO \sma \CP^{c +
d}_c$, where $d = p^k(p-1) -1$. This leads to the proof of \Cref{main1}. In
\Cref{sub:Tate}, we  prove \Cref{main-tate}.

\subsection{A splitting of $\EO \sma \CP^\infty_c$} 
\label{sub:splitting} 
Now we extend the `$\text{free} \hspace{.7pt} \oplus  \hspace{.7pt} \text{finite}$'
decomposition of $\K_*[\Cp]$-modules in \Cref{KCPinfty-decomposition}
 to a `$\text{free } \oplus \text{ finite}$'  decomposition of
$\E_*[\Cp]^\sigma$-comodules (see \Cref{ECPinfty-decomposition}) using the fact
that a free $\E_*[\Cp]^\sigma$-module  is injective relative to $\E_*$ (see
\Cref{relinj}). 
\begin{defn}[\cite{relatively-injective}]
Let  $\Smr$ be a subring of a commutative ring $\R$. An $\R$-module
$\mr{M}$ is \emph{$(\R, \Smr)$-injective} (injective relative to $\Smr$) if
every $\R$-module  injection $\Mmr \hookrightarrow \Nmr$ that splits  in the
category of $\Smr$-modules  also splits in the category of $\R$-modules.
\end{defn}
\begin{prop}
\label{relinj}
A free $\ECpsigma$-module $\Fmr$ is $(\ECpsigma, \E_*)$-injective. 
\end{prop}
\begin{proof}
Hochschild \cite[Lemma 1]{relatively-injective} showed that
$\Hom_{\Smr}(\Rmr,\Amr)$ is $(\R, \Smr)$-injective when $\Smr$ is a subring
of $\Rmr$ and $\Amr$ is an $\Smr$-module.
Since
\[ \ECpsigma\cong\Hom_{\E_*}(\ECpsigma, \E_*)\]
as a $\E_*[\Cp]^{\sigma}$-module, the result follows. 
\end{proof}
\begin{lem}
\label{K-to-E-splitting}
Suppose $\mr{Q}^\E$ is a free $\ECp$-module whose underlying $\E_*$-module is
free and there is a $\KCp$-module splitting
 \[ \mr{Q}^\E/\mathfrak{m} \cong \Fmr^\K \oplus \Mmr^\K \]  
where $\Fmr^\K$ is free and $\Mmr^\K$ is finitely generated. Then there is an $\ECp$-module splitting
\[  \mr{Q}^{\E} \cong \Fmr^\E \oplus \Mmr^\E,\]
where $\Fmr^\E$ is free and $\Mmr^\E$ is finitely generated.
\end{lem}
\begin{proof} Let $ \Fmr^\E$ be a
 free $\ECp$-module  such that $\Fmr^\E/
\mathfrak{m} \cong \Fmr^\K$. Fix an $\ECp$-basis $\mathcal{B} = \{ {\sf b}_i: i
\in \NN \} $ of $\Fmr^\E$. Let 
\[ 
\begin{tikzcd}
\iota_\K \colon \Fmr^\K \rar[shift left, hook, ""] & \lar[shift left, two heads, ] \mr{Q}^\E/\mathfrak{m} \rcolon \pi_\K
\end{tikzcd}
\]
denote the $\KCp$-maps that split $\Fmr^\K$ from $\mr{Q}^\E/\mathfrak{m}$. We
will now show that the maps $\iota_K$ and $\pi_\K$ can be lifted to
$\ECp$-module maps $\iota_\E$ and $\pi_\E$  
\[ 
\begin{tikzcd}
\Fmr^\E  \dar[two heads,"\pi_1"' ] \rar[shift left, hook, dashed,  "\iota_\E"] & \lar[shift left, two heads, dashed, "\pi_\E"] \mr{Q}^\E \dar[two heads, "\pi_2"] \\
\Fmr^\K \rar[shift left, hook, "\iota_\K"] & \lar[shift left, two heads, "\pi_\K"] \mr{Q}^\E/\mathfrak{m}
\end{tikzcd}
\] 
such that $\pi_\E \circ \iota_\E$ is the identity. 

The $\ECp$-linear map $\iota_\E$ can be defined by sending ${\sf b}_i$ to an
arbitrary lift of $\iota_\K(\pi_1({\sf b}_i))$.  Since $\E_*$ is Noetherian,
$\Fmr^\E$ and $\mr{Q}^\E$ are both $\E_*$-free and $\iota_\K$ is injective, we
conclude  that
$\iota_\E$ is also injective (essentially from the Krull intersection theorem). The freeness of $\Fmr^\E$ and
$\mr{Q}^\E$ as $\E_*$-modules also implies the existence of an
$\E_*$-linear map $\pi_\E$ such that $\pi_\E \circ \iota_\E$ is the identity. By \Cref{relinj}, $\pi_\E$ can be chosen so that it is $\ECp$-linear.
\end{proof}
\begin{cor}
\label{ECPinfty-decomposition}
 For any $c \in \ZZ$, there exists an  $\ECp$-module isomorphism
\[ \E_*\CPinfty_c \cong \Fmr_c^\E \oplus \Mmr_c^\E,\] 
where  $\Fmr_c^\E$ is  free  and $\Mmr_c^\E$ is finitely generated  as $\ECp$-modules.
\end{cor}

Next, we use the relative Adams spectral sequence for the map
$\EO \to \E$ to show that the algebraic splitting of
\Cref{ECPinfty-decomposition} originates from an  $\EO$-module splitting of  $\EO\sm \CPinfty_c$. 
\begin{defn}
An $\EO$-module $\mathcal{X}$ is \emph{relatively projective} (resp.
\emph{relatively free}) if $\EEO_*\mathcal{X}$ is a projective (resp. free)
$\E_*$-module. 
\end{defn}
\begin{ex}
The spectrum $\E$ is a relatively free $\EO$-module.
\end{ex}
\begin{ex} 
When $\Xmr$ is an even spectrum $\EO \sma \Xmr$ is
a relatively free $\EO$-module.

\end{ex} 
\begin{thm}[{\cite[Corollary 3.4]{Devinatz-homotopy-fixed-point-spectra}}]
\label{thm:rel-ASS} Suppose that $\mathcal{X}$
and $\mathcal{Y}$ are $\EO$-modules such that $\mathcal{X}$ is finite and relatively
projective. The Adams spectral sequence  relative to the map $\EO\to\E$
\begin{equation} \label{rel-ASS}
 \Emr^{s,t}_2 \coloneqq
    \Ext^{s,t}_{\EEO_*\E}(\EEO_*\mathcal{X},\EEO_*\mathcal{Y})
    \Rightarrow
    \pi_{t-s}\EOmod(\mathcal{X},\mathcal{Y}),
\end{equation}
is strongly convergent.
\end{thm}
The relative Adams spectral sequence was developed by Baker and Lazarev
\cite{BakerLazarev}. Devinatz \cite{Devinatz-homotopy-fixed-point-spectra}
defined the homotopy fixed points $\E^{\hG}$ for $\mr{G}$ an arbitrary closed
subgroup of $\Aut(\Gamma)$. He identified the $\E_2$-page of the relative Adams
spectral sequence for $\E^{\hG}\to \E$ with the group cohomology of $\mr{G}$ and
showed that for every $\E^{\hG}$-module the relative Adams spectral sequence is
strongly convergent and that there is a uniform horizontal vanishing line
independent of the $\E^{\hG}$-module. Rognes \cite{Rognes-galois} developed the
theory of Galois extensions of ring spectra and reinterpreted the results of
Devinatz  by concluding that the map $\EO\to \E$ is Galois (see \cite[Theorem
5.4.4]{Rognes-galois}). Let
\[ \DEO(-) \coloneqq \EOmod(-,\EO) \colon \EO\text{-module} \to
\EO\text{-module} \] denote the relative Spanier-Whitehead dual functor.
\begin{prop}{\cite[Proposition 6.4.7]{Rognes-galois}}
\label{E-self-dual} $\DEO(\E) \simeq \E$.
\end{prop}

\begin{lem} \label{relASS-collapse}
The edge homomorphism for the relative Adams spectral sequence
\begin{equation} \label{edge}
\begin{tikzcd}
\pi_\ast \EOmod(\mathcal{X}, \mathcal{Y}) \rar &  \EEOEcomod(\EEO_*\mathcal{X},\EEO_*\mathcal{Y})
\end{tikzcd}
\end{equation}
 is an isomorphism if
\begin{enumerate}[(I)]
\item \label{relASS-collapseI}  $\mathcal{Y}= \E$ and $\mathcal{X}$ is a finite relatively-projective $\EO$-module, or if
\item $\mathcal{X} = \E$ and $\mathcal{Y}$ is an arbitrary $\EO$-module. 
\end{enumerate}
\label{rel-ASS-collapse}
\end{lem}
\begin{proof} 
When $\mathcal{Y}=\E$ and $\mathcal{X}$ is finite, the edge map \eqref{edge} is
an isomorphism, as $\EEO_*\E$ is a cofree $\EEO_*\E$-comodule and the spectral
sequence \eqref{rel-ASS} is concentrated in the zero line.

When $\mathcal{X}=\E$  is a finite self-dual $\EO$-module (see
\Cref{E-self-dual}) and therefore, 
 \[
    \mathcal{X} \sm_{\EO} \mathcal{Y}
    \simeq \DEO(\mathcal{X})\sm_{\EO}\mathcal{Y}
    \simeq \EOmod(\mathcal{X},\mathcal{Y}).
\]
Since $\EEO_*\E$ is $\E_*$-free, we have a Kunneth isomorphism
\[
\EEO_*(\DEO(\E) \sm_{\EO} \mathcal{Y}) \cong \EEO_*(\E\sm_{\EO} \mathcal{Y}) \cong \EEO_*(\E) \otimes_{\E_*} \EEO_*\mathcal{Y}
\] 
and consequently $\EEO_*(\DEO(\E) \sm_{\EO} \mathcal{Y})$ is cofree as an
$\EEO_*\E$-comodule. Therefore, the spectral sequence \eqref{rel-ASS} is
concentrated on the zero line and the edge homomorphism \Cref{edge} is an
equivalence as desired.
\end{proof}

As a consequence of the convergence of the relative Adams spectral sequence, we
deduce that $\KEO_*(-)$ detects equivalences of $\EO$-modules:
\begin{cor}
\label{KEO-detects-equivalences}
Suppose that $\mathcal{X}$ and $\mathcal{Y}$ are relatively projective
$\EO$-modules. An $\EO$-module map 
\[ 
\begin{tikzcd}
f\colon \mathcal{X}\rar & \mathcal{Y}
\end{tikzcd}
\]
is a weak equivalence if and only if
\[
\begin{tikzcd}
 f_*\colon \KEO_*\mathcal{X} \rar["\cong"] & \KEO_*\mathcal{Y}
\end{tikzcd}
\] is a  $\K_*$-isomorphism. 
\end{cor}
\begin{proof}
By the Nakayama lemma, if $f_*\colon \KEO_*\mathcal{X}\to\KEO_*\mathcal{Y}$ is
an isomorphism, then so is $f_*\colon \EEO_*\mathcal{X}\to\EEO_*\mathcal{Y}$.
Thus $f$ induces an isomorphism of relative Adams $\E_2$-pages. It follows that
$f$ induces an isomorphism on $\pi_*$. The converse is obvious. 
\end{proof}
\begin{defn}
We say an  $\EO$-module $\mathcal{X}$ is \emph{freely filtered} if there is a
filtration
\[ \mathcal{X}^{(0)}\to \mathcal{X}^{(1)}\to \cdots\to \mathcal{X}\]
such that
\begin{enumerate}
\item $\mathcal{X}^{(0)}$ is contractible,
\item each $\mathcal{X}^{(d)}$ is a compact relatively free $\EO$-module, 
\item the map $\mathcal{X}^{(d)} \to \mathcal{X}^{(d+1)}$ is an $\EO$-module map, and
\item $\colim_{d} \mathcal{X}^{(d)} \simeq \mathcal{X}$.
\end{enumerate}
\end{defn}
\begin{ex} 
For any even spectrum $\Xmr$ of finite type, $\EO \sma \Xmr$ is a freely filtered
$\EO$-module where the filtration is induced by the skeletal filtration of
$\Xmr$. 
\end{ex}

\begin{lem}
\label{EO-splitting-free}
Suppose that  $\mathcal{X}$ is a freely filtered  $\EO$-module with a  $\KCp$-module splitting
\[\KEO_*(\mathcal{X})/\mfrak \cong \Fmr^\K\oplus \Mmr^\K,\] where $\Fmr$ is 
free with countable basis. Then there exists an $\EO$-module splitting
\[ \mathcal{X}\simeq \Fcal\vee \Mcal\] 
where $\Fcal$ is a free $\E$-module such that $\KEO_*\Fcal \cong \Fmr$ and
$\Mcal$ is an $\EO$-module such that $\KEO_*\Mcal\cong \Mmr$.
\end{lem}
\begin{proof} 
Let us present the free $\KCp$-module $\Fmr$ as   
\[ \Fmr^\K := \bigoplus_{\mathcal{B}}\KCp \cdot {\sf b}_i  \]
using a basis 
$\mathcal{B} := \{ {\sf b}_i: 0 \leq i < n \leq \infty \}$. 
Let $\Fmr_i^\K := \KCp \cdot {\sf b}_i \subset \Fmr^\K$. Let $\iota^\K$ and
$\pi^\K$ denote the inclusion and the projection map for the split-summand
corresponding to $\Fmr_0^\K$. Because $\Fmr_0^\K$ is a finite $\KCp$-module,
there exists a solution to the diagram of $\KCp$-modules 
\begin{equation} \label{solutionK}
\begin{tikzcd}
\Fmr_0^\K \dar[dashed, shift right, "\iota^\K_{\ell}"'] \rar[equal]&\Fmr_0^\K \rar[equal] \dar[dashed, shift right, "\iota^\K_{\ell+1}"'] & \dots  \rar[equal]& \Fmr_0^\K \ar[d,shift right,"\iota^\K"'] \\
\KEO_*\mathcal{X}^{(\ell)} \rar \ar[u, shift right,  "\pi^\K_{\ell}"']  &\KEO_*\mathcal{X}^{(\ell+1)} \rar \ar[u, shift right,  "\pi^\K_{\ell+1}"']&\dots \rar & \KEO_*\mathcal{X} \ar[u, shift right, "\pi^\K"'] 
\end{tikzcd}
 \end{equation}
for some $\ell \gg 0$. Using \Cref{K-to-E-splitting}, we can extend
\Cref{solutionK} to a diagram of $\ECp$-modules (or equivalently
$\EEO_*\E$-comodules) 
\begin{equation} \label{solutionE}
\begin{tikzcd}
\EEO_*\E \dar[ shift right, "\iota^\E_{\ell}"'] \rar[equal]&\EEO_*\E \rar[equal] \dar[ shift right, "\iota^\E_{\ell+1}"'] & \dots  \rar[equal]& \EEO_*\E\ar[d,shift right,"\iota^\E"'] \\
 \EEO_*\mathcal{X}^{(\ell)} \rar \ar[u, shift right, "\pi^\E_{\ell}"']  &\EEO_*\mathcal{X}^{(\ell+1)} \rar \ar[u, shift right,  "\pi^\E_{\ell+1}"']&\dots \rar & \EEO_*\mathcal{X} \ar[u, shift right, "\pi^\E"'].
\end{tikzcd}
\end{equation}
By \Cref{relASS-collapse}, the maps $\iota_{j}^\E$ and $\pi_{j}^\E$ can be
realized in the homotopy category of $\EO$-modules 
\[ 
\begin{tikzcd}
\upiota_j\colon \E \rar[shift right] & \lar[shift right] \mathcal{X}^{(j)} \rcolon \uppi_i
\end{tikzcd}
\]
for all $i \geq \ell$. Since $\colim_i \mathcal{X}^{(i)} \simeq \mathcal{X}$,
we conclude that there exists an $\EO$-module $\mathcal{M}$ such that 
\[ \KEO_* \mathcal{M} \cong \bigoplus_{i =1}^{n-1}\KCp \cdot {\sf b}_i  \oplus \Mmr^\K  \]
 and an $\EO$-module equivalence
 \[ \mathcal{X} \simeq \E \vee \Mcal. \]
Note that $\mathcal{M}$ can be filtered as 
\[ 
 \Mcal^{(0)}\to \Mcal^{(1)}\to \cdots\to \Mcal,
\]
where $\Mcal^{(i)} = \cofiber(\upiota_{ \ell + i})$ and $\EEO_*\Mcal^{(i)}$
is $\E_*$-free. Thus our results follow from an inductive argument. 
\end{proof}
Combining \Cref{KCPinfty-decomposition} and \Cref{EO-splitting-free}, we get:  
\begin{cor} \label{splittingEOCPc}For any $c \in \ZZ$, there exists an $\EO$-module splitting 
\[ \EO \sma \CP^\infty_c \simeq \mathcal{F}_c \vee \mathcal{M}_c,\]
where $\mathcal{F}_c \simeq \bigvee_{\NN} \E$ and $\mathcal{M}_c$ is a compact $\EO$-module. 
\end{cor}
\Cref{main2} is the case $c=0$ in \Cref{splittingEOCPc}. 
\subsection{An $\EO$-Thom isomorphism} 
\label{sub:EOThom}
Recall from from the proof of \Cref{grK*CPinfty} that the summand  $\Mmr_{0}^{ \
\gr\K}$ of  $\gr\K_*\CP^\infty_+$ factors through $\gr\K_*
\CP^{\hat{\upbeta}_k}_+$, where \[ \hat{\upbeta}_k :=p^{k}( p -1) \] is the
smallest multiple of $p$ which is greater than $\bk = |\Pmr_k^{p-1}|/2 =
(p-1)(p^k -1)$. This algebraic factorization can be leveraged to obtain the
following result. 
\begin{lem} 
\label{M-factor-skeleton} 
For all $c \in \ZZ$,  the inclusion map  $\sfs''\colon \Mcal_{pc}
\hookrightarrow \EO \sma \CP^\infty_{pc} $ of the splitting in
\Cref{splittingEOCPc} can be chosen so that the inclusion of $\Mcal_{pc}$
factors through $\EO \sma \CP^{\hat{\upbeta}_k + pc}_{pc}$ 
\[ 
 \begin{tikzcd}
 &  \Mcal_{pc} \dar["\sfs''"]  \ar[ld, dashed, "\sfsbar'"']\\
  \EO \sma \CP^{\hat{\upbeta}_k + pc}_{pc} \rar[hook, "\sk"] & \EO \sma \CP^\infty_{pc} 
 \end{tikzcd}
\] 
in the homotopy category of $\EO$-modules. 
\end{lem} 
\begin{proof} 
We present the argument for $c =0$. The general argument follows from the
$\Bcal(k)$-module isomorphism of \Cref{Bk-thom-iso}. 

We start with an arbitrary splitting of $\EO \sma \CPinfty_+$, where  
\[ \begin{tikzcd} 
\Fcal_0 \rar["\sfs"] & \EO \sm \CPinfty_+ \rar["\sff"] & \Fcal_0 
\end{tikzcd}\]
\[ \begin{tikzcd}
\Mcal_0 \rar["\sfs'"] & \EO \sm \CPinfty_+ \rar["\sff'"] & \Mcal_0,
\end{tikzcd}\]
are the inclusion and projection maps of the components.
We will modify this splitting to produce a new one which satisfies our conclusion.

The free summand $\Fmr_0$ of the $\Bcal(k)$-module $\mr{H}_*\CPinfty$ surjects
onto $\mr{H}_*\CPinfty_{\obk}$ under the coskeletal map. Let $\Rcal$ be the
fiber of $\cosk \circ \sfs \circ \sff$. By \Cref{leading-term-of-grK(X)},
the composite
\[ \begin{tikzcd}
\EO \sm \CPinfty_+ \rar["\sff"] & \Fcal_0 \rar["\sfs"] & \EO \sm \CPinfty_+ \rar["\cosk"] & \EO \sm \CPinfty_{\obk}
\end{tikzcd}\]
induces a surjection on $\KEO_*(-)$. Hence, the map  
\[\begin{tikzcd}
\Rcal \rar & \EO \sm \CPinfty_+,
\end{tikzcd}\]
 induces an
injection on $\KEO_*(-)$. Because $\EO \sm \CP^{\obk-1}_+$ is the fiber of
$\cosk$, there is a map
\[\begin{tikzcd}
 \alpha\colon \Rcal \rar & \EO \sm \CP^{\obk -1}_+. 
\end{tikzcd} \]
Since $\KEO_*(\cosk \circ \sfs \circ \sff) = \KEO_*(\cosk)$, the  induced map 
\[\begin{tikzcd}
\KEO_*\alpha \colon \KEO_*\Rcal \rar &\KEO_*( \EO \sm \CP^{\obk -1}_+) \cong  \K_* \CP^{\obk -1}
\end{tikzcd}\]
is an isomorphism of $\K_*$-modules. By \Cref{KEO-detects-equivalences},
$\alpha$ is a weak equivalence of $\EO$-modules. 

Because $\sff \circ \sfs'$ is null, the composite
$(\cosk \circ \sfs \circ \sff)\circ \sfs'$ 
is also null and therefore there is a map  $\Mcal_0 \to \Rcal$. Now let
$\sfsbar' \colon\Mcal_0 \to \mathcal{R} \overset{\simeq}\longrightarrow \EO \sm
\CP^{\obk -1}_+ $ and define $\sfs''$ to be the composite   
\[ 
\begin{tikzcd}
\Mcal_0 \rar["\sfsbar'"] & \EO \sm \CP^{\obk -1}_+  \rar & \EO \sm \CPinfty_+ .
\end{tikzcd}
\]
It is easy to see that $\KEO_*(\cofiber(\sfs'')) \cong \KEO_* \Fcal_0$. Thus, by \Cref{EO-splitting-free}, 
\[ \cofiber(\sfs'') \simeq \Fcal_0 \simeq \bigvee_{\NN} \E \] 
 so $\cofiber(\sfs'')$ is a split summand of $\EO \sma \CPinfty_+$.
\end{proof}
\begin{lem}
\label{Bk-thom-iso}
The $\HF_p$-Thom isomorphism $\H_{*+ c}\CPinfty_{c} \cong \H_{*}\CPinfty_+$ is a map of
$\Bcal(k)$-modules if $c$ is divisible by $p$. 
\end{lem}
\begin{proof} 
Since the action of $\mathcal{B}(k)$ on $\H_*(-)$ is obtained from the action of
$\mathcal{B}(k)$ on $\H^*(-)$ via the isomorphism \[ \H_*(-) \cong
\Hom_{\FF_p}( \H^*(-), \FF_p),\] it is enough to show the $\H\FF_p$-Thom
isomorphism in cohomology 
\[  \uptau \colon \H^{*}\CPinfty_+  \overset{\cong}\longrightarrow \H^{*+
c}\CPinfty_{c} \]
is a map of $\mathcal{B}(k)$-modules when $p$ divides $c$. Note that
$\Pmr_k(u_c) = c x^{p^k -1}\cdot u_c$, where  $u_c \in \H^*\CPinfty_{pc}$ is the
Thom class. Therefore,  $\Pmr_k(u_c) = 0$ when $c$ is divisible by $p$ and 
\[ \uptau(\Pmr_k(x^i)) = \Pmr_k(x^i) \cdot u_c = \Pmr_k(x^i \cdot u_c) = \Pmr_k(\uptau(x^i)), \]  
as desired. 
\end{proof}

Our next goal is to prove the following $\EO$-Thom isomorphism. 
\begin{thm} \label{EO-Thom-iso} There exists an equivalence of $\EO$-modules  
\begin{equation} \label{EO-thom-equiv}
  \EO \sma \CPinfty_{\sfr} \simeq \EO \sma \Sigma^{2 \sfr} \CPinfty_+ ,
  \end{equation}
where $\sfr =\Theta( \upgamma^{\obk -1}, \EO)$. 
\end{thm}
\begin{proof} 
Since $\sfr = \order( \upgamma^{\obk -1}, \EO)$, we have an $\EO$-Thom
isomorphism
\begin{equation} \label{EOthom-truncate}
\EO \sma \CP_{\sfr}^{\obk -1 + \sfr} \simeq \EO \sma \Sigma^{2\sfr} \CP^{\obk -1}_+. 
\end{equation}
Using \Cref{splittingEOCPc} and \Cref{M-factor-skeleton}, we construct the composite map 
\[ 
\begin{tikzcd}
\sf{c}\colon\Mcal_\sfr \rar["\sfsbar'"] & \EO \sma \CP_{\sfr}^{\obk -1 + \sfr} \rar["\simeq"]  & \EO \sma\Sigma^{2\sfr} \CP_{+}^{\obk -1 } \rar &  \EO \sma \Sigma^{2\sfr} \CP_{+}^{\infty }.
\end{tikzcd}
\]
 By \Cref{thom-skeleton}, \Cref{EOthom-truncate} respects the Atiyah-Hirzebruch
filtration, therefore  we analyze the image of the map induced by $\sfc$ on AHSS
calculating  $\KEO_*(-)$-groups.  

Since  $\Mcal_{\sfr}$ does not have a natural Atiyah-Hirzebruch filtration, we induce one 
\[ 
\begin{tikzcd}
\Mcal_\sfr^{(0)} \rar & \dots \rar & \Mcal_\sfr^{(\obk -1)} = \Mcal_\sfr^{(\obk) } = \dots =  \Mcal_\sfr
\end{tikzcd}
\]   
 by pulling back the Atiyah-Hirzebruch filtration on $\EO \sma \Sigma^{2\sfr} \CP_{+}^{\infty }$.  The associated graded of the correponding filtration on $ \KEO_* \Mcal_{\sfr}$, denote it by $\gr  \KEO_* \Mcal_{\sfr}$,  is a  $\grKCp$-module, and the induced map 
 \begin{equation} \label{grKEO-f}
 \begin{tikzcd}
\gr \KEO_*({\sf c}) \colon \gr  \KEO_* \Mcal_{\sfr} \rar &  \gr \K_* \Sigma^{2 \sfr}\CPinfty_+,
 \end{tikzcd}
 \end{equation}
is an injection with image $\Sigma^{2 \sfr}\Mmr_0^{ \ \gr \K}$. This is because
of the explicit description of $\Sigma^{2 \sfr}\Mmr_0^{ \ \gr \K}$ as in the
proof of \Cref{grK*CPinfty} and \Cref{Bk-thom-iso}.   The cofiber  of ${\sf c}$
also admits a filtration
\[ 
\begin{tikzcd}
\cofiber({\sf c}) := \colim_n\{ \cofiber({\sf c})^{(0)} \rar & \cofiber({\sf c})^{(1)} \rar & \dots \} ,
\end{tikzcd}
\]   
where $\cofiber({\sf c})^{(i)} = \mr{Cofiber}(\Mcal_\sfr^{(i)} \to \EO \sma
\Sigma^{2\sfr} \CP^i_+)$, such that the associated graded of the induced
filtration on $\KEO_*\cofiber({\sf c})$ is a free  $\grKCp$-module isomorpic to
$\Sigma^{2\sfr}\Fmr_0^{\ \gr \K}$. Thus, by \Cref{lift-free-summand-from-grKCp}
and \Cref{EO-splitting-free}
\[\cofiber({\sf c}) \simeq \Sigma^{2 \sfr} \Fcal_0\] 
is a split summand of $\EO \sma \Sigma^{2 \sfr} \CPinfty_+$. Therefore, we have
an equivalence 
\[ 
\EO \sma \Sigma^{2 \sfr} \CPinfty_+ \simeq \Mcal_{\sfr} \vee \Sigma^{2r} \Fcal_0 \simeq \Mcal_{\sfr} \vee \Fcal_{\sfr} \simeq \EO \sma \CPinfty_{\sfr} 
\]
as desired. 
\end{proof}

\begin{cor} \label{cor:order-reduce} $\order(\upgamma, \EO) = \order( \upgamma^{\obk -1}, \EO).$
\end{cor}
\begin{proof} 
Immediate from \Cref{EO-Thom-iso} and the fact that $\Mmr(\sfr
\hspace{1pt} \upgamma) \cong \CP_\sfr^{\infty}.$
\end{proof}
\Cref{main1.5} follows from \Cref{cor:order-reduce} because 
$\order(\EO, \upgamma^{\obk-1})$ divides \linebreak $\order(\SS_{p}, \upgamma^{\obk -1}) =
p^{p^k -1}$ (see \Cref{complex-James-periodicity-formula}). 

\subsection{The $\mr{S}^1$-Tate fixed points of $\EO$}\label{sub:Tate}
We shift our focus to identifying the $\Smr^1$-Tate spectrum of $\EO$. If a
group $\Gmr$ acts on a spectrum $\Xmr$, the Tate spectrum $\Xmr^{\tG}$ of $\Xmr$ is the
cofiber of the norm map from the homotopy orbits spectrum $\Xmr_{\hG}$ of $X$ to
the homotopy fixed points spectrum $\Xmr^{\hG}$ of $X$
\[ \Xmr^{\tG} \coloneqq \cofiber(\mr{Nm}\colon \Xmr_{\hG} \to \Xmr^{\hG}).\] 
If $\Gmr = \Smr^1$ and the action on $X$ is trivial, \cite{GMTate} gives an
alternate description of the Tate fixed points:
\begin{equation}
 \label{tate-lim}
 \Xmr^{\tSo} \simeq\lim_{c \to \infty} \Sigma^2\R \sm \CPinfty_{-c}.
\end{equation}
Inverse limits do not commute with $\pi_*(-)$. Instead, there
is a short exact sequence
 \[\begin{tikzcd}[column sep = small]
0\rar &
\lim^1_{c}(\R_*\CPinfty_{-c}) \rar&
\pi_*\R^{\tSo}\rar &
\lim_{c}(\R_*\CPinfty_{-c}) \rar&
0
\end{tikzcd}\]
called the ``Milnor sequence''.

\begin{proof}[Proof of \Cref{main-tate}]
By \Cref{M-factor-skeleton} we can choose a splitting of $\EO \sma \CPinfty_{i \obk}$ for all $i \in \ZZ$ such that we have the commutative diagram 
\begin{equation}
\label[diagram]{sexseq-of-EO-module-limit-diagrams}
\begin{tikzcd}
\cdots      \rar &
\Mcal_{(i-1) \obk} \rar["0"] \dar &
\Mcal_{i \obk}   \rar["0"]\dar &
\Mcal_{(i+1) \obk} \rar["0"]\dar &
\cdots\\
\cdots \rar &
\EO\sm\CPinfty_{(i-1)\obk}  \rar\dar &
\EO\sm\CPinfty_{i\obk}      \rar \dar &
\EO\sm\CPinfty_{(i+1)\obk}  \rar\dar&
\cdots\\
\cdots                \rar&
\Fcal_{(i-1) \obk}  \rar &
\Fcal_{i \obk}      \rar &
\Fcal_{(i+1) \obk}  \rar &
\cdots\\
\end{tikzcd}
\end{equation}
such that the horizontal maps in the top row are null maps, in the middle row
are coskeletal collapse maps.

The horizontal maps in the bottom row induce surjections in homotopy. To
see this, we filter $\Fcal_{i \obk}$ using the
Atiyah-Hirzebruch filtration of $\CPinfty_{i \obk}$. Let $\gr \KEO_* \Fcal_{i
\obk}$ denote the associated graded of the induced filtration on $\KEO_*\Fcal_{i
\obk}$. From \Cref{grK*CPinfty}, we observe that the composite 
\[ 
\begin{tikzcd}
f_i\colon \Fcal_{i \obk} \rar & \EO \sma  \CPinfty_{i \obk} \rar  &  \EO \sma  \CPinfty_{(i+1) \obk}  \rar & \Fcal_{(i+1) \obk}
\end{tikzcd}
 \]
induces a surjection in $\gr \KEO_*(-)$, and therefore a surjection in
$\KEO_*(-)$. By the Nakayama Lemma, $f_i$ induces surjection on
$\EEO_*(-)$. By studying the induced map of relative Adams spectral
sequence, we conclude that $f_i$ is a surjection. 

Consequently, we get a six-term exact sequence
\[\begin{tikzcd}
0\rar&
\lim_i\pi_*\Mcal_{ - i \obk}\rar&
\lim_i \EO_*\CPinfty_{-i \obk}\rar&
\lim_i \pi_*\Fcal_{-i \obk}
    \ar[dll, out= east, in=west, looseness=1.9, overlay,]\\
&
\limone_i \pi_*\Mcal_{- i \obk}\rar&
\limone_i \EO_*\CPinfty_{-i\obk}\rar&
\limone_i \pi_*\Fcal_{-i\obk}\rar&
0.
\end{tikzcd}\]
From \Cref{sexseq-of-EO-module-limit-diagrams},
\begin{align*}
\lim\nolimits_i\pi_*\susp^{-2i\sfr}\Mcal_0  &= 0\\
\limone_i \pi_*\susp^{-2i\sfr}\Mcal_0 &=0 \\
\lim\nolimits_i^1 \pi_* \Fcal_{-i\sfr} &= 0
\end{align*}
and
\[\lim\nolimits_i \Fcal_{-i\obk} \simeq    \prod_{-\infty< k <\infty} \E.\]
Thus, $\limone_i \EO_*\CPinfty_{-i\obk} = 0$ and 
\[ \EO^{\mr{t} \mr{S}^1} \simeq \lim\nolimits_i  \EO \sma \CPinfty_{-i} \simeq
\lim\nolimits_i \EO \sma \CPinfty_{-i\obk} \simeq \lim\nolimits_i\Fcal_{-i\obk}
\simeq  \prod_{-\infty< k <\infty} \E \] 
as desired. 
\end{proof}

\bibliographystyle{amsalpha}
\bibliography{EO-orientation}
\end{document}